\documentclass[12pt]{article}
\usepackage{tikz}
\usetikzlibrary{calc,through,backgrounds}
\usepackage{graphicx,enumerate}
\usepackage{subfigure} 
\usepackage[justification = centering, labelsep =period]{caption} 

\usepackage{multicol}

\usepackage{bm,amssymb,amsthm,amsfonts,amsmath,latexsym,cite,color, psfrag,graphicx,ifpdf,pgfkeys,pgfopts,xcolor,extarrows}

\newtheorem{theo}{Theorem}[section]

\newtheorem{lem} [theo]{Lemma}
\newtheorem{coro}[theo]{Corollary}
\newtheorem{prop}[theo]{Proposition}

\renewcommand{\thefigure}{\arabic{section}.\arabic{figure}}
\allowdisplaybreaks

\makeatletter \@addtoreset{equation}{section}
\@addtoreset{theo}{}\makeatother

\usepackage{hyperref}

\renewcommand{\theequation}{\arabic{section}.\arabic{equation}}
\renewcommand{\thesection}{\arabic{section}}
\renewcommand{\thetheo}{\arabic{section}.\arabic{theo}}

\renewcommand{\arraystretch}{1.3}
\def\qed{\hfill \rule{4pt}{7pt}}
\def\pf{\noindent {\it Proof.} }
\def\S{  \mathfrak{S}}
\def\N{  \mathrm{N}}
\def\Min{  \mathrm{Min}}
\def\Max{  \mathrm{Max}}
\def\GL{   \mathrm{GL}  }
\def\Char{   \mathrm{char}  }
\def\tr{   \mathrm{tr}  }
\def\cc {\mathcal{C}}
\def\FL { \mathcal{FL}}
\def\FLa {\mathcal{FL}_{C^{1}}}
\def\FLb {\mathcal{FL}_{C^{2}}}
\begin{document}
\begin{center}
{\Large\bf Zero-one dual characters of flagged Weyl modules}

\vskip 6mm
{\small   }
Peter L. Guo, Zhuowei Lin, Simon C.Y. Peng

\end{center}

\begin{abstract}
We prove a  criterion of when the  dual character $\chi_{D}(x)$  of the flagged Weyl module associated to a diagram  $D$ in the grid $[n]\times [n]$ is zero-one, that is, the coefficients of monomials in $\chi_{D}(x)$ are either  0 or 1. This settles  a conjecture proposed   by M{\'e}sz{\'a}ros--St. Dizier--Tanjaya.
Since Schubert polynomials and key polynomials occur as  special cases  of dual flagged Weyl characters, our approach  provides  a new and unified proof of  known  criteria for zero-one Schubert/key polynomials   due to Fink--M{\'e}sz{\'a}ros--St. Dizier and Hodges--Yong, respectively.

\end{abstract}

\vskip 3mm

\noindent {\bf Keywords:}   flagged Weyl module, dual character, Schubert polynomial, key polynomial,  zero-one polynomial

\vskip 3mm

\noindent {\bf AMS Classifications:} 05E10,  05E14, 05A19, 14N15

\section{Introduction}

The goal of this paper is  to confirm a conjectured  criterion  for zero-one dual characters of flagged Weyl modules   due to
M{\'e}sz{\'a}ros,   St. Dizier  and   Tanjaya \cite[Conjecture 3.9]{meszaros2021principal}. 
A polynomial is called zero-one (or, multiplicity-free) if   its coefficients are equal to either $0$ or $1$. It is well known that dual flagged Weyl characters encompass  Schubert and key polynomials as special cases. So our result  leads to a unified proof  for  previously known criteria respectively   for zero-one Schubert polynomials by Fink, M{\'e}sz{\'a}ros and St. Dizier \cite{fink2021zero} and   zero-one key polynomials by   Hodges and Yong \cite{hodges2023multiplicity}.

Motivations of  investigating  zero-one dual characters of flagged Weyl modules are multifold. It was proved by Fink,  M{\'e}sz{\'a}ros and   St. Dizier \cite{2018Schubert} that the supports of the dual character  of a flagged Weyl module are exactly the same as the integer  points in its Newton polytope. Hence   a zero-one dual flagged Weyl character  is completely  determined by the integer points in its Newton polytope. 

As aforementioned, Schubert and key polynomials appear as special cases of dual characters of flagged Weyl models, see for example  \cite{2018Schubert} and references therein. A criterion for zero-one Schubert polynomials was found  by   Fink, M{\'e}sz{\'a}ros and  St. Dizier \cite{fink2021zero}, whose proof makes use of  Magyar's orthodontia formula for Schubert polynomials \cite{magyar1998schubert, mestrans}. By the work of Knutson and Miller \cite{KM-2},  Schubert polynomials are the multidgree polynomials of matrix Schubert varieties. Particularly, zero-one Schubert polynomials correspond to multiplicity-free matrix Schubert  varieties, and in this case,  one could capture the entire  information of K-polynomials (namely, Grothendieck polynomials)  from that of  zero-one Schubert polynomials \cite{Brion, CCFM, EL, Knutson-1, MSSD-2, PS-2022}.  Moreover, for a permutation  whose associated Schubert polynomial is zero-one,   the CDG generators of the  defining ideal    of its matrix Schubert variety form a diagonal Gr\"obner basis  \cite{Klein, HPW}.


A criterion (first announced in \cite[Theorem 4.10]{HY-1}) for zero-one key polynomials was proved by Hodges and Yong \cite{hodges2023multiplicity} by employing   two combinatorial models for  key polynomials:  the quasi-key model by   Assaf and Searles \cite {AS-1} and  the  Kohnert diagram model by Kohnert   \cite{Koh}. The algebraic motivation of studying zero-one key polynomials stems from the classification of Levi-spherical Schubert varieties \cite{HY-1,GHY}.

Another notable  subfamily of dual flagged Weyl characters are   for northewest diagrams. This  subfamily also includes Schubert and key polynomials as special cases.
It was shown by Armon,     Assaf,  Bowling and   Ehrhard \cite{AABE}  that the dual   flagged Weyl character for a northwest diagram coincides with the Kohnert polynomial defined  in \cite{AS-22}. As remarked in \cite{hodges2023multiplicity}, the techniques they developed for zero-one key polynomials do not seem to extend to    zero-one dual flagged Weyl characters for northwest diagrams.

We proceed by describing   the criterion for zero-one dual characters of flagged Weyl modules, as  stated   in  \cite[Conjecture 3.9]{meszaros2021principal}.
Let  $[n]=\{1,2,\ldots, n\}$.  A subset $D$ of boxes in the   grid $[n]\times [n]$   is called a diagram. The   flagged Weyl module $\mathcal{M}_{D}$ associated to a diagram $D$  is a module of the group $B$ of invertible upper-triangular matrices over $\mathbb{C}$ \cite{kraskiewicz1987foncteurs, kraskiewicz2004schubert, magyar1998schubert}, see Section \ref{FWM} for the precise definition.  The dual character of $\mathcal{M}_{D}$, denote $\chi_{D}(x)$, is a polynomial in $x_1,\ldots, x_n$, which specifies to the Schubert polynomial $\mathfrak{S}_w(x)$ when $D$ is the Rothe diagram of a permutation $w$, and to the key polynomial $\kappa_\alpha(x)$ when $D$ is the skyline diagram of a composition $\alpha$. The flagged Weyl modules have inspired considerable recent research interests, see for example \cite{AABE, FG-1, fan2022upper, 2018Schubert, fink2021zero, upper, MT-23, meszaros2021principal, mestrans}.

Following  the terminology in \cite{meszaros2021principal}, a {\it multiplicitous} configuration refers to one of the six instances of configurations listed in  Figure  \ref{avoid}.
Here, a ``$\times$'' indicates the absence of a box, a blank square  indicates the presence of a  box, and a ``$*$'' indicates no restriction
 on the presence or absence of a box.
  \begin{figure}[h]
        \centering
		\begin{tikzpicture}
			\centering
                \draw[black] (-5.5,1) -- (-5,0.5);
                \draw[black] (-5.5,0.5) -- (-5,1);
                \draw[black] (-5,1) -- (-4.5,0.5);
                \draw[black] (-5,0.5) -- (-4.5,1);
                \draw[black] (-5,0.5) -- (-4.5,0);
                \draw[black] (-5,0) -- (-4.5,0.5);
                \draw[black] (-5,0.5) -- (-4.5,1);
                \draw[black,step=0.5] (-5.5,1) grid (-4.5,-1);
                \node at (-5.25,-0.25) {$\boldsymbol{*}$};
                \node at (-5.25,-0.75) {$\boldsymbol{*}$};
                \node at (-4.75,-0.75) {$\boldsymbol{*}$};
                \draw[black] (-3,1) -- (-3.5,0.5);
                \draw[black] (-3,0.5) -- (-3.5,1);
                \draw[black] (-3.5,0.5) -- (-3,0);
                \draw[black] (-3,0.5) -- (-3.5,0);
                \draw[black] (-2.5,0.5) -- (-3,0);
                \draw[black] (-3,0.5) -- (-2.5,0);
                \draw[black,step=0.5] (-3.5,1) grid (-2.5,-1);
                \node at (-4.75+2,-0.75) {$\boldsymbol{*}$};
                \draw[black] (-1.5,1) -- (-0.5,0);
                \draw[black] (-1.5,0) -- (-0.5,1);
                \draw[black] (-1,0) -- (-0.5,0.5);
                \draw[black] (-1,1) -- (-0.5,0.5);
                \draw[black] (-1.5,0.5) -- (-1,1);
                \draw[black] (-1.5,0.5) -- (-1,0);
                \draw[black,step=0.5] (-1.5,1) grid (-0.5,-1);
                \node at (-0.75,-0.75) {$\boldsymbol{*}$};
                \draw[black] (0.5,1) -- (1.5,0);
                \draw[black] (1,1) -- (0.5,0.5);
                \draw[black] (1,0) -- (1.5,0.5);
                \%filldraw [lightgray] (1,0) rectangle (0.5,-0.5);
                \draw[black,step=0.5] (0.5-0.0001,1) grid (1.5,-1);
                \node at (-0.75+1.5,-0.75) {$\boldsymbol{*}$};
                \node at (-0.75+2,-0.75) {$\boldsymbol{*}$};
                \draw[black] (2.5,1) -- (3.5,0);
                \draw[black] (3,1) -- (2.5,0.5);
                \draw[black] (3.5,0.5) -- (3,0);
                \draw[black] (3.5,0) -- (2.5,-1);
                \draw[black] (3.5,-0.5) -- (3,0);
                \draw[black] (3,-1) -- (2.5,-0.5);
                \draw[black,step=0.5] (2.5-0.0001,1) grid (3.5,-1);

                \draw[black] (5.5,1) -- (4.5,0);
                \draw[black] (5.5,0.5) -- (5,1);
                \draw[black] (4.5,0.5) -- (5,0);
                \draw[black] (5,0) -- (5.5,-0.5);
                \draw[black] (5.5,0) -- (4.5,-1);
                \draw[black] (5,-1) -- (4.5,-0.5);
                \draw[black,step=0.5] (4.5-0.0001,-1) grid (5.5,1);
                \node at  (-5,-1.5) {$(A)$};
                \node at  (-3,-1.5) {$(B)$};
                \node at  (-1,-1.5) {$(C)$};
                \node at  (1,-1.5) {$(D)$};
                \node at  (3,-1.5) {$(E)$};
                \node at  (5,-1.5) {$(F)$};
		\end{tikzpicture}
        \caption{Multiplicitous configurations.}
        \label{avoid}
    \end{figure}
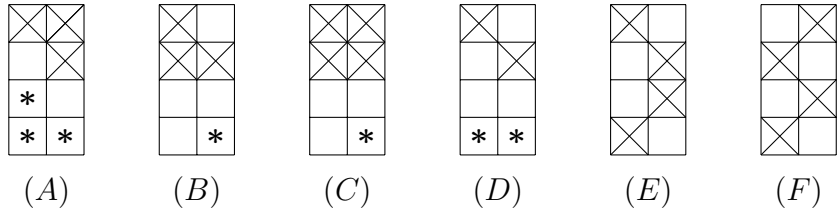

A diagram $D$ is called {\it multiplicitous} if it contains one of multipliticous configurations as a subdiagram, up to possibly swapping the order of the columns. This means that there exist row indices $i_1<i_2<i_3<i_4$ and column indices $j_1<j_2$ such that the subdiagram of $D$, which is restricted  to rows $\{i_1, i_2, i_3, i_4\}$ and columns  $\{j_1, j_2\}$, is    either  a multiplicitous  configuration  or a configuration obtained from a multiplicitous  configuration  by swapping   its two columns. We call a diagram  $D$ {\it multiplicity-free} if it is not    multiplicitous. So a multiplicity-free diagram avoids a total of 12 configurations. 

In light of  an observation in \cite[Theorem 5.8]{fink2021zero},  it is readily  verified  that if the dual character $\chi_{D}(x)$ is  zero-one, then $D$ must be  multiplicity-free \cite[Proposition 3.11]{meszaros2021principal}. This paper aims  to prove the reverse direction  which appears as \cite[Conjecture 3.9]{meszaros2021principal}. 

\begin{theo}[{\cite[Conjecture 3.9]{meszaros2021principal}}]\label{MMain}
    If  $D$ is a multiplicity-free  diagram, then the dual character   $\chi_{D}(x)$ is  zero-one.
      
\end{theo}

Combining Theorem  \ref{MMain} with the above mentioned \cite[Proposition 3.11]{meszaros2021principal} gives a full criterion for zero-one dual Weyl characters.

\begin{coro}
The dual character   $\chi_{D}(x)$ is  zero-one if and only if   $D$ is    a multiplicity-free diagram.     
\end{coro}

The above corollary on the one hand specializes to the criterion for zero-one Schubert polynomials \cite[Theorem 1]{fink2021zero} when $D$ is the Rothe diagram of a permutation, and on the other hand to the criterion  for zero-one key polynomials \cite[Theorem 1.1]{hodges2023multiplicity}  when $D$ is the skyline diagram of a composition. 


This paper is structured  as follows. In Section \ref{FWM}, we review the construction of  the flagged Weyl module $\mathcal{M}_{D}$, as well as some facts about its dual character $\chi_{D}(x)$. Particularly, it will be seen that the coefficient of a monomial  $x^a$ in $\chi_{D}(x)$ is equal to the dimension of the eigenspace in $\mathcal{M}_{D}$ corresponding to  $x^a$. Section \ref{sect3-rr} is devoted to several lemmas concerning the properties  of multiplicity-free diagrams.  In Section \ref{MMain-pp}, we  finish the proof of Theorem \ref{MMain}   which will be achieved by showing that when a diagram $D$ is multiplicity-free, each  eigenspace in $\mathcal{M}_{D}$    has dimension exactly equal to one.  The arguments are combinatorial.

\subsection*{Acknowledgements}
This work was  supported by the National Natural Science Foundation of China (Nos. 11971250, 12371329), and the Fundamental Research Funds for the Central Universities (No. 63243072).

\section{Dual characters of flagged Weyl modules}\label{FWM}

In this section, we   give an overview of some basic information about  flagged Weyl modules. We mainly follow the notation in \cite{2018Schubert, fink2021zero}.

For the square grid $[n]\times [n]$, let us  use $(i, j)$ to denote the box in row $i$ and column $j$ in the matrix coordinate. Recall that a diagram $D$ is a subset of boxes in $[n]\times [n]$. Write $D=(D_{1}, D_{2},\ldots, D_{n})$, where, for $1\leq j\leq n$, $D_j$ stands for the $j$-th column of $D$. 
We shall often represent $D_j$ by  a subset of $[n]$, that is,   $i\in D_j$
 if and only if the box $(i,j)$ belongs to $D$. 
 For example, the diagram in Figure \ref{fig:diaex} can be written  as $D=(\{2,3,4\},\emptyset,\{1,2\},\{3\})$.
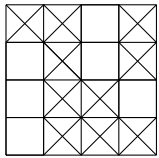
\begin{figure}[h]
    \centering
	\begin{tikzpicture}
		\centering
		\draw[black,step=0.5] (-1,-1) grid (1,1);
            \draw[black] (-1,1) -- (1,-1);
            \draw[black] (-1,0.5) -- (-0.5,1);
            \draw[black] (-1+0.5,0.5) -- (-0.5+0.5,1);
            \draw[black] (-1+1.5,0.5) -- (-0.5+1.5,1);
            \draw[black] (-0.5+0.5,0.5) -- (-1+0.5,1);
            \draw[black] (-0.5+1.5,0.5) -- (-1+1.5,1);
            \draw[black] (-1+0.5,0) -- (-0.5+0.5,1-0.5);
            \draw[black] (-1+1.5,0) -- (-0.5+1.5,1-0.5);
            \draw[black] (-0.5+0.5,0) -- (-1+0.5,1-0.5);
            \draw[black] (-0.5+1.5,0) -- (-1+1.5,1-0.5);
            \draw[black] (-1+0.5,0-0.5) -- (-0.5+0.5,0);
            \draw[black] (-0.5,0) -- (0 ,-0.5);
            \draw[black] (-0.5+0.5,-0.5) -- (0+0.5 ,0);
            \draw[black] (-1+0.5,0-1) -- (-0.5+0.5,-0.5);
            \draw[black] (-0.5,0-0.5) -- (0 ,-0.5-0.5);
            \draw[black] (-0.5+0.5,-0.5-0.5) -- (0+0.5 ,0-0.5);
            \draw[black] (0,-0.5) -- (0.5, -1);
            \draw[black] (0.5,-1) -- (1,-0.5);
		\draw[black,step=0.5] (-1,-1) grid (-0.5,0.5);
		\draw[black,step=0.5] (0,0) grid (0.5,1);
		\draw[black,step=0.5] (0.5,-0.5) grid (1,0);
	\end{tikzpicture}
    \caption{A diagram in $[4]\times [4]$.}
    \label{fig:diaex}
\end{figure}

For two $k$-element  subsets $R=\{r_1< \cdots<r_k\}$ and $S= \{s_1< \cdots<s_k\} $ of $[n]$, write  $R\leq S$  if  $r_i\leq s_i$ for $1\leq i\leq k$. This defines a partial order, called the {\it Gale order}, on all $k$-element subsets of $[n]$. For a  given  subset $S$ of $[n]$, the collection of subsets $R\leq S$   constitutes the  basis of a matroid, called { Schubert matroid},    on the ground set $[n]$ \cite{2018Schubert}. 
It is   worth mentioning that the Gale  order plays an important role in the study of the positroid decomposition  of (positive) Grassmannians \cite{KLS, Post}. 
For two diagrams $C=(C_{1}, \ldots, C_{n})$ and $D=(D_{1}, \ldots,D_{n})$,  write  $C\leq D$ if $C_{j}\leq D_{j}$ for each $1\leq j\leq n$. 

Let $\mathrm{GL}(n,\mathbb{C})$ be the general linear group of $n\times n$ invertible matrices over $\mathbb{C}$, and $B$  the Borel subgroup of $\mathrm{GL}(n,\mathbb{C})$ consisting of all  upper-triangular matrices. Let $Y$ be the upper-triangular matrix of variables $y_{ij}$ with $1\leq i\leq j\leq n$:
\begin{equation*}
	Y=
	\begin{bmatrix}
		y_{11}&y_{12}&\ldots&y_{1n}\\
		0&y_{22}&\ldots&y_{2n}\\
		\vdots&\vdots&\ddots&\vdots\\
		0&0&\ldots&y_{nn}
	\end{bmatrix}.
\end{equation*}
Denote by  $\mathbb{C} [Y]$   the linear space of polynomials over $\mathbb{C}$ in the variables  $\{y_{ij}\}_{i\leq j}$. Define the (right) action of $B$   on $\mathbb{C} [Y]$ by   $f(Y)\cdot b=f(b^{-1}\cdot Y)$, where $b\in B$ and  $f\in \mathbb{C} [Y]$. For two subsets $R$ and $S$ of $[n]$ with the same cardinality, we use $Y^{R}_{S}$ to represent  the submatrix of $Y$ obtained by restricting to rows indexed by $R$ and columns indexed by $S$. It is not hard to check that $\det\left(Y^{R}_{S}\right)\neq 0$ if and only if $R\leq S$.
Given two diagrams $C=(C_{1}, \ldots, C_{n})$ and  $D=(D_{1}, \ldots, D_{n})$ with $C\leq D$, denote 
\[
\mathrm{det}\left(Y_{D}^C\right)=\prod_{j=1}^{n}\mathrm{det}\left(Y_{D_{j}}^{C_{j}}\right).
\]
The flagged Weyl module $\mathcal{M}_{D}$  for  $D$ is  a subspace of $\mathbb{C} [Y]$ defined by
\begin{align*}	\mathcal{M}_{D}=\mathrm{Span}_{\mathbb{C}}\left\{\mathrm{det}\left(Y_D^C\right)\colon C\leq D\right\},
\end{align*}
which  is a $B$-module with the action inherited from the action of $B$ on $\mathbb{C} [Y]$.
 
Let $X$ be the diagonal matrix with diagonal entries $x_{1},\ldots,x_{n}$.
The character of $\mathcal{M}_{D}$ is defined by
\[\mathrm{char}(\mathcal{M}_{D})(x_{1},\ldots, x_{n})=\mathrm{tr}(X\colon \mathcal{M}_{D}\rightarrow \mathcal{M}_{D}).\] 
It can be directly checked  that  $\mathrm{det}\left(Y_{D}^C\right)$  is an
eigenvector of $X$ with eigenvalue
\[\prod_{j=1}^n\prod_{i\in C_j}x_i^{-1}.\]
The   dual character is  defined as 
\begin{align*}
\chi_{D}(x):=\mathrm{char}(\mathcal{M}_{D})(x_{1}^{-1},\ldots,x_{n}^{-1}).
\end{align*}
Hence   the coefficient of a monomial $x^a$
in $\chi_{D}(x)$  is equal to the dimension of the corresponding  eigenspace:
\begin{align}\label{UUUP}
[x^a]\,\chi_{D}(x)=\dim \mathrm{Span}_{\mathbb{C}}\left\{\mathrm{det}\left(Y_{D}^C\right)\colon  C\leq D, \, x^C=x^a\right\},
\end{align}
where
\[
x^{C}=\prod_{j=1}^{n}\prod_{i\in C_{j}} x_{i}.
\]
Particularly, the set of monomials appearing in  $\chi_{D}(x)$ is   $ \left\{x^C\colon C\le D\right\}$.

\begin{prop}\label{PPPo}
The dual character $\chi_{D}(x)$ is   zero-one  if and only if each eigensapce on the right-hand side of \eqref{UUUP} has dimension one.  
\end{prop}
    
We end this section with  a combinatorial expression   for the generator $\mathrm{det}\left(Y_{D}^C\right)$.
A  {\it flagged filling} of a diagram  $D=(D_1,\ldots, D_n)$   is a filling of the boxes of $D$ with positive integers   such that  
\begin{itemize}
	\item[(1)] each box receives exactly one integer, and the entries in each column are distinct;
		
  \item[(2)] the entry in the box $(i,j)\in D$ cannot exceed $i$.
\end{itemize}

Let $F=(F_1,\ldots, F_n)$ be a flagged filling of $D$, where, for $1\leq j\leq n$, $F_j$ denotes the $j$-th column of $F$.  Define the inversion number $\mathrm{inv}(F)$ of $F$ as follows.  Let $w=w_1\cdots w_m$ be the word obtained by reading the entries of $F_j$ from top to bottom. Denote by   $\mathrm{inv}(F_j)$   the inversion number of $w$,
namely, 
\[
\mathrm{inv}(F_j)=\#\{(r, s)\colon 1\leq r<s\leq m, \ w_r>w_s\}.
\]
In the case when  $D_j$ is empty, we adopt the convention that $\mathrm{inv}(F_j)=0$. 
Define
\[
\mathrm{inv}(F)=\mathrm{inv}(F_1)+\cdots +\mathrm{inv}(F_n).
\]
For the  flagged filling   depicted  in Figure \ref{fig:enter-label-UU22}, 
\begin{figure}[ht]
    \centering
	$F$~=~\begin{tikzpicture}[baseline={([yshift=-.5ex]current bounding box.center)}]
		\centering
		\draw[black,step=0.5] (-1,-1) grid (1,1);
            \draw[black] (-1,1) -- (-0.5,0.5);
            \draw[black] (-1,0.5) -- (-0.5,1);
            \draw[black] (-1+.5,1) -- (-0.5+.5,0.5);
            \draw[black] (-1+.5,0.5) -- (-0.5+.5,1);
            \draw[black] (-1+1.5,1) -- (-0.5+1.5,0.5);
            \draw[black] (-1+1.5,0.5) -- (-0.5+1.5,1);
            \draw[black] (-1+1.5,1-.5) -- (-0.5+1.5,0.5-.5);
            \draw[black] (-1+1.5,0.5-.5) -- (-0.5+1.5,1-.5);
            \draw[black] (-1+1,1-.5) -- (-0.5+1,0.5-.5);
            \draw[black] (-1+1,0.5-.5) -- (-0.5+1,1-.5);
            \draw[black] (-1+1,1-1) -- (-0.5+1,0.5-1);
            \draw[black] (-1+1,0.5-1) -- (-0.5+1,1-1);
            \draw[black] (-1+1,1-1.5) -- (-0.5+1,0.5-1.5);
            \draw[black] (-1+1,0.5-1.5) -- (-0.5+1,1-1.5);
		\draw[black,step=0.5] (-1,-1) grid (-0.5,0.5);
        \draw[black,step=0.5] (-0.5,0) rectangle (0,0.5);
        \draw[black,step=0.5] (-0.5,0) rectangle (0,-0.5);
        \draw[black,step=0.5] (-0.5,-1) rectangle (0,-0.5);
		\draw[black,step=0.5] (0,1) grid (0.5,0.5);
		\draw[black,step=0.5] (0.5,-1) grid (1,0);
		\node at (-0.75,-0.75) {1};
		\node at (-0.75,-0.25) {3};
		\node at (-0.75,0.25) {2};
		\node at (0.25,0.75) {1};
		\node at (0.75,-0.25) {3};
         \node at (0.75,-0.75) {2};
        \node at (-0.25,-0.75) {3};
		\node at (-0.25,-0.25) {1};
		\node at (-0.25,0.25) {2};
	\end{tikzpicture}
    \caption{A flagged filling.}
    \label{fig:enter-label-UU22}
\end{figure}
The column reading words are $231, 213, 1$ and $32$ from left to right, and so we have $\mathrm{inv}(F)=2+1+0+1=4$.
Assign a weight to  $F$  in the following way: 	
\begin{align*}	y^{F}=\prod_{(i, j)\in D}y_{c_{ij}i},
\end{align*}
 where $c_{ij}$ is the entry of $F$ filled  in the box $(i, j)$. For example, the flagged  filling in   Figure \ref{fig:enter-label-UU22} gives
 \[
y^{F}=y_{11}\,y_{22}^2\,y_{13}\,y_{33}^2\,y_{14}\,y_{24}\,y_{34}.
 \]

For   $C\leq D$,  
let $\mathcal{F}_{D}(C)$ denote the set  of  flagged fillings    $F=(F_1,\ldots, F_n)$ of $D$ such that for $1\leq j\leq n$,  the set of entries in  $F_j$ is exactly $C_j$.
The following proposition appears as \cite[Lemma 2.2]{upper}.
 
 \begin{prop}\label{ABC-1}
For $C\leq D$,  we have 
\begin{align}\label{OOPP}	\mathrm{det}\left(Y_{D}^C\right)=\sum_{F\in \mathcal{F}_{D}(C)}\mathrm{sgn}(F)\cdot y^{F},
	\end{align}
	where $\mathrm{sgn}(F)=(-1)^{\mathrm{inv}(F)}$.
\end{prop}

As explained in  \cite[Remark 2.3]{upper}, the formula in \eqref{OOPP}	may not be cancellation-free, that is, there may exist distinct fillings in $\mathcal{F}_{D}(C)$ that contribute the same weight, but have opposite signs.

\section{Multiplicity-free diagrams}\label{sect3-rr}

In this section, we investigate  the distribution of boxes in a multiplicity-free diagram. This will be a step stone in the proof that each eigenspace for a multiplicity-free diagram has dimension exactly equal to one.

Clearly,  the dual character $\chi_{D}(x)$ is independent of the order of columns of $D$. We say that two diagrams 
are {\it equivalent} if one can be obtained from  the other by reordering the   columns. 
We shall designate  a particular representative  in the class of  diagrams equivalent to $D$. Such a representative  is called the {\it normalization} of $D$. 

\subsection{Normalization of diagrams}

Let $D=(D_1,\ldots, D_n)$ be a diagram. The {  normalization}, denoted $\mathrm{Norm}(D)$, of $D$ is the unique diagram equivalent to $D$ by rearranging the columns in the following way. 
For $1\leq j\leq n$, let $\overline{D}_j=([n]\setminus D_j)\cup \{\infty\}$. 
For  $ {D}_{j_1}\neq  {D}_{j_2}$, suppose that 
\[
\overline{D}_{j_1}=\{a_1<\cdots<a_s<a_{s+1}=\infty\},\ \ \ \ \ 
\overline{D}_{j_2}=\{b_1<\cdots<b_t<b_{t+1}=\infty\}.
\]
Define $ {D}_{j_1}\prec {D}_{j_2}$ if  $\overline{D}_{j_1}<_{\mathrm{lex}}\overline{D}_{j_2}$ 
in the lexicographic order, that is, there exists some index $k$ such that 
$a_i=b_i$ for $1\leq i<k$ and $a_k<b_k$. 
The usage of $\infty$ in $\overline{D}_{j}$ is to ensure that any two distinct  columns of $D$  are comparable.
Now $\mathrm{Norm}(D)$ is defined by rearranging the columns of  $D$ increasingly from left to right with respect to $\prec$. 
See Figure \ref{normalization}
for an illustration of  the normalization of a diagram.
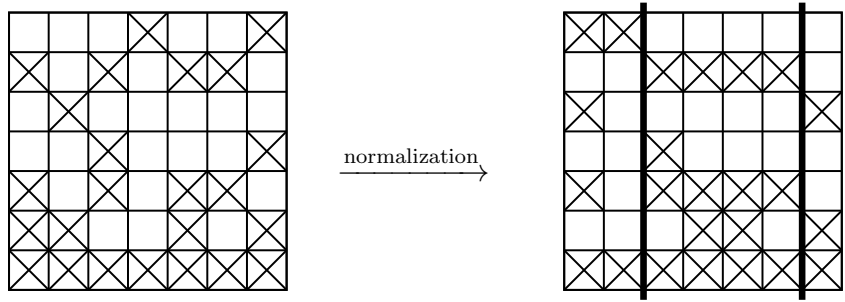
\begin{figure}[h]
    \centering

\tikzset{every picture/.style={line width=0.75pt}} 

\begin{tikzpicture}[x=0.75pt,y=0.75pt,yscale=-1,xscale=1]

\draw  [draw opacity=0] (0,0) -- (140,0) -- (140,140) -- (0,140) -- cycle ; \draw   (20,0) -- (20,140)(40,0) -- (40,140)(60,0) -- (60,140)(80,0) -- (80,140)(100,0) -- (100,140)(120,0) -- (120,140) ; \draw   (0,20) -- (140,20)(0,40) -- (140,40)(0,60) -- (140,60)(0,80) -- (140,80)(0,100) -- (140,100)(0,120) -- (140,120) ; \draw   (0,0) -- (140,0) -- (140,140) -- (0,140) -- cycle ;
\draw  [draw opacity=0] (280,0) -- (420,0) -- (420,140) -- (280,140) -- cycle ; \draw   (300,0) -- (300,140)(320,0) -- (320,140)(340,0) -- (340,140)(360,0) -- (360,140)(380,0) -- (380,140)(400,0) -- (400,140) ; \draw   (280,20) -- (420,20)(280,40) -- (420,40)(280,60) -- (420,60)(280,80) -- (420,80)(280,100) -- (420,100)(280,120) -- (420,120) ; \draw   (280,0) -- (420,0) -- (420,140) -- (280,140) -- cycle ;
\draw    (0,20) -- (20,40) ;
\draw    (0,40) -- (20,20) ;

\draw    (0,80) -- (20,100) ;
\draw    (0,100) -- (20,80) ;

\draw    (0,100) -- (20,120) ;
\draw    (0,120) -- (20,100) ;

\draw    (20,40) -- (40,60) ;
\draw    (20,60) -- (40,40) ;

\draw    (20,100) -- (40,120) ;
\draw    (20,120) -- (40,100) ;

\draw    (40,20) -- (60,40) ;
\draw    (40,40) -- (60,20) ;

\draw    (40,80) -- (60,100) ;
\draw    (40,100) -- (60,80) ;

\draw    (60,0) -- (80,20) ;
\draw    (60,20) -- (80,0) ;

\draw    (80,20) -- (100,40) ;
\draw    (80,40) -- (100,20) ;

\draw    (80,80) -- (100,100) ;
\draw    (80,100) -- (100,80) ;

\draw    (80,100) -- (100,120) ;
\draw    (80,120) -- (100,100) ;

\draw    (100,20) -- (120,40) ;
\draw    (100,40) -- (120,20) ;

\draw    (40,60) -- (60,80) ;
\draw    (40,80) -- (60,60) ;

\draw    (100,80) -- (120,100) ;
\draw    (100,100) -- (120,80) ;

\draw    (0,120) -- (20,140) ;
\draw    (0,140) -- (20,120) ;

\draw    (20,120) -- (40,140) ;
\draw    (20,140) -- (40,120) ;

\draw    (60,120) -- (80,140) ;
\draw    (60,140) -- (80,120) ;

\draw    (40,120) -- (60,140) ;
\draw    (40,140) -- (60,120) ;

\draw    (80,120) -- (100,140) ;
\draw    (80,140) -- (100,120) ;

\draw    (100,120) -- (120,140) ;
\draw    (100,140) -- (120,120) ;

\draw    (120,120) -- (140,140) ;
\draw    (120,140) -- (140,120) ;

\draw    (280,120) -- (300,140) ;
\draw    (280,140) -- (300,120) ;

\draw    (300,120) -- (320,140) ;
\draw    (300,140) -- (320,120) ;

\draw    (340,120) -- (360,140) ;
\draw    (340,140) -- (360,120) ;

\draw    (320,120) -- (340,140) ;
\draw    (320,140) -- (340,120) ;

\draw    (360,120) -- (380,140) ;
\draw    (360,140) -- (380,120) ;

\draw    (380,120) -- (400,140) ;
\draw    (380,140) -- (400,120) ;

\draw    (400,120) -- (420,140) ;
\draw    (400,140) -- (420,120) ;

\draw    (120,0) -- (140,20) ;
\draw    (120,20) -- (140,0) ;

\draw    (120,60) -- (140,80) ;
\draw    (120,80) -- (140,60) ;

\draw    (120,100) -- (140,120) ;
\draw    (120,120) -- (140,100) ;

\draw    (280,0) -- (300,20) ;
\draw    (280,20) -- (300,0) ;

\draw    (280,40) -- (300,60) ;
\draw    (280,60) -- (300,40) ;

\draw    (280,80) -- (300,100) ;
\draw    (280,100) -- (300,80) ;

\draw    (300,0) -- (320,20) ;
\draw    (300,20) -- (320,0) ;

\draw    (320,20) -- (340,40) ;
\draw    (320,40) -- (340,20) ;

\draw    (320,80) -- (340,100) ;
\draw    (320,100) -- (340,80) ;

\draw    (320,60) -- (340,80) ;
\draw    (320,80) -- (340,60) ;

\draw    (340,20) -- (360,40) ;
\draw    (340,40) -- (360,20) ;

\draw    (340,80) -- (360,100) ;
\draw    (340,100) -- (360,80) ;

\draw    (340,100) -- (360,120) ;
\draw    (340,120) -- (360,100) ;

\draw    (360,20) -- (380,40) ;
\draw    (360,40) -- (380,20) ;

\draw    (360,80) -- (380,100) ;
\draw    (360,100) -- (380,80) ;

\draw    (360,100) -- (380,120) ;
\draw    (360,120) -- (380,100) ;

\draw    (380,20) -- (400,40) ;
\draw    (380,40) -- (400,20) ;

\draw    (380,80) -- (400,100) ;
\draw    (380,100) -- (400,80) ;

\draw    (400,40) -- (420,60) ;
\draw    (400,60) -- (420,40) ;

\draw    (400,100) -- (420,120) ;
\draw    (400,120) -- (420,100) ;

\draw[line width=2.5pt] (320,-5) --(320,145);
\draw[line width=2.5pt] (400,-5) --(400,145);
\draw (164,65.4) node [anchor=north west][inner sep=0.75pt]    {$\xrightarrow{\mathrm{normalization}}$};

\end{tikzpicture}
\caption{An illustration of  normalization.}
\label{normalization}
\end{figure}

A diagram $D$ is called {\it normalized} if $D=\mathrm{Norm}(D)$. Given  a normalized diagram $D$, we partition the columns of   $D$ into distinct regions, according to the   positions of the  first crossing  in each column. Here, the crossings in each column are counted from top to bottom. Precisely, columns  of $D$ belong to the same region if the first crossings in these columns lie in the same row. For example,   the right normalized diagram in Figure \ref{normalization}   is partitioned  into three  regions, as divided  by lines in boldface.  

\subsection{Columns in the same region}\label{subss}

In this subsection, we analyze the distribution of boxes of a diagram $D=(D_1,\ldots, D_n)$ in the same region. We could always assume that each column in $[n]\times [n]$ has at least two ``$\times$'' since otherwise one may embed $D$ into a larger  grid $[m]\times [m]$ with $m>n$. With this in mind, we distinguish the columns of $D$ into three types, according to the number of boxes below the second crossing. 
For $1\leq j\leq n$, we say that 
\begin{itemize}
    \item $D_j$ is of  $\mathrm{Type\ I} $ if there is no box of $D_j$ below the second crossing;
    
    \item $D_j$ is of $\mathrm{Type\ II} $ if there is exactly one box of $D_j$ below the second crossing;

\tikzset{every picture/.style={line width=0.75pt}} 
   
   \item $D_j$ is of $\mathrm{Type\ III} $ if there are at least two boxes  of $D_j$ below the second crossing.
\end{itemize}
Define   the {\it signature} of  $D_j$   as  the number of boxes of $D_j$ lying between the first   and the second crossings. For example, the third   column in the right  diagram in Figure \ref{normalization} is a  Type II column with signature equal to 1. When $D$ is a normalized diagram, the signatures of   columns in the same region are weakly increasing from left to right.

\begin{lem} \label{longest}
Let $D=(D_1,\ldots, D_n)$ be a normalized multiplicity-free diagram. Consider the columns of $D$ in the same region. Suppose that a column $D_j$ does not reach the maximum signature   among columns in the   region. Then $D_j$    must  be of  $\mathrm{Type\ I} $.
\end{lem}

\begin{proof}
Since $D_j$ does not reach the maximum signature, there exists a column $D_{j'}$  with $j<j'$
in the same region such that $D_{j'}$ has a larger signature than $D_j$. Assume that  the first two crossings    in column $j$ are at the positions   $(i_1,j)$ and $(i_2,j)$. Since $D_j$ and $D_j'$ lie in the same region,  the first crossing   in column $j'$ lies in row $i_1$. Suppose that the second crossing in column $j'$ lies in row $i_3$. Then we have $i_3>i_2$, and so the box $(i_2, j')$ belongs to $D_j'$. See Figure \ref{FF-22} for an illustration.
\begin{figure}[h]
    \centering
        \begin{tikzpicture}
        \node at (0.25,-1.65+2) {\vdots}; 
        \node at (0.75+.5,-1.65+2) {\vdots}; 
        \draw[black,step=0.5] (0,0) grid (0.5,-0.5);
        \draw[black] (0,0) -- (0.5,-0.5);
        \draw[black] (0,-0.5) -- (0.5,0);
        \draw[black,step=0.5] (0+1-.001,0.000001) grid (0.5+1,-0.5);
        \draw[black] (1+.5,0) -- (0.5+.5,-0.5);
        \draw[black] (0.5+.5,0) -- (1+.5,-0.5);
        \draw (0,-1.5) rectangle (0.5,-1);
        \draw[black] (0,-1) -- (0.5,-1.5);
        \draw[black] (0,-1.5) -- (0.5,-1);
        \draw (0+1,-1.5) rectangle (1.5,-1);
        \draw[black] (1,-1-1) -- (1.5,-1.5-1);
        \draw[black] (1,-1.5-1) -- (1.5,-1-1);
        \node at (-0.25,-0.25) {$i_1$};
        \node at (0.25,-1.65+1) {\vdots}; 
        \node at (0.75+.5,-1.65+1) {\vdots}; 
        \node at (0.25,-1.65-0.5) {\vdots}; 
        \node at (0.75+.5,-1.65) {\vdots}; 
        \draw[black,step=0.5] (0,-2.5-0.5) grid (0.5,-3-0.5);
        \draw[black,step=0.5] (0.999,-2) grid (0.5+1,-2.5);
        \node at (0.25,-3.25) {$\boldsymbol{*}$};
        \node at (-0.25,-3.25) {$i$};
        \node at (-0.25,-1.25) {$i_2$};
        \node at (-0.25,-2.25) {$i_3$};
        \node at (0.25,0.25+0.5) {$j$};
        \node at (0.75+0.5,0.25+0.5) {$j'$};
    \end{tikzpicture}
    \caption{An illustration for the proof of Lemma \ref{longest}.}\label{FF-22}
\end{figure}

Let $(i,j)$ be any box below $(i_2, j)$. We conclude that $(i,j)$ does not belong to $D_j$ since otherwise the subdiagram of $D$, which is restricted to rows $\{i_1, i_2, i\}$  and columns 
$\{j, j'\}$, would form  the configuration obtained by swapping the columns of $(A)$ in Figure \ref{avoid}.  This implies   that  $D_j$ is a column of Type I.
\end{proof}

From Lemma \ref{longest}, we see that if there are $\mathrm{Type\ II} $ or $\mathrm{Type\ III} $ columns  occurring   in the same region, then they must be columns with  the maximum signature.
We discuss this in detail in the  next lemma.

\begin{lem}\label{BGTR}
Let  $D=(D_1,\ldots, D_n)$ be a normalized multiplicity-free diagram. Consider the columns in the same region. Then
\begin{itemize}
    \item[(1)] $\mathrm{Type\ II} $ columns and $\mathrm{Type\ III} $ columns cannot occur simultaneously;

    \item[(2)] If there is a $\mathrm{Type\ II} $   column occurring, then all $\mathrm{Type\ II} $   columns   have  the same configuration  (In this case,  the region starts with some $\mathrm{Type\ I} $ columns, followed by some copies of a $\mathrm{Type\ II} $ column, see the left picture in Figure \ref{fig:enter-label-T} for an illustration);

    \item[(3)] 
    If there is a $\mathrm{Type\ III} $   column occurring, then this is the only $\mathrm{Type\ III} $ column  in the  region (In this case,  the region starts with some $\mathrm{Type\ I} $ columns, followed by a single $\mathrm{Type\ III} $ column, see the right  picture  in Figure \ref{fig:enter-label-T} for an illustration).
\end{itemize}
\end{lem}

\begin{proof}
(1) Suppose to the contrary that there are both a $\mathrm{Type\ II} $  column (say, $D_{j_1}$) and a $\mathrm{Type\ III} $ column (say, $D_{j_2}$)  occurring in the same region. We only consider the case for $j_1<j_2$, and the arguments for  $j_1>j_2$ are the same.  Since 
column $j_1$ and column $j_2$ are in the same region,  their  first crossings   lie in the same row, say row $i_1$. 
By Lemma \ref{longest}, $D_{j_1}$ and $D_{j_2}$ have the same signature, and so 
the second  crossings in   column $j_1$ and column $j_2$ are also in the same row, say  row $i_2$. 
Let $(i, j_1)$ be the unique  box of $D_{j_1}$ below row $i_2$.  Since $D_{j_2}$ is of $\mathrm{Type\ III} $, it has at least two boxes below row $i_2$. We have the following observation.
\begin{itemize}
    \item [(O)]  For $i_2<h<i$, there is no box $(h, j_2)$ belonging to $D_{j_2}$, since otherwise the subdiagram, restricted to rows $\{i_2, h, i\}$ and columns $\{j_1, j_2\}$, would become   the configuration obtained by swapping the columns of $(A)$   in Figure \ref{avoid}.
\end{itemize}

The arguments are divided into two cases.

Case 1: The box $(i, j_2)$ belongs to $D_{j_2}$.  By the observation in (O), there exists another  box $(\ell, j_2)$ of $D_{j_2}$ with $\ell>i$. So the subdiagram, restricted to rows $\{i_1,i_2, i, \ell\}$ and columns $\{j_1, j_2\}$, is the configuration by swapping the columns of $(C)$ in Figure \ref{avoid}, leading to a contradiction.

Case 2: The box $(i, j_2)$ does not belong to $D_{j_2}$. Still, by the observation in  (O), one can choose a box $(\ell,j_2)$ of $D_{j_2}$ with $\ell>i$.
In this case, the subdiagram, restricted to rows $\{i_2, i, \ell\}$ and columns $\{j_1, j_2\}$, forms the configuration    $(A)$ in Figure \ref{avoid},  yielding  a contradiction. This concludes the assertion in (1). 

(2) Let $D_{j_1}$ and $D_{j_2}$ with $j_1<j_2$ be two columns of $\mathrm{Type\ II}$. As in (1), let $(i_1, j_1), (i_2, j_1)$ (respectively, $(i_1, j_2), (i_2, j_2)$) be boxes containing  the first two crossings in column $j_1$ (respectively, column $j_2$).
The unique box of $D_{j_1}$ below row $i_2$ must be in the same row as the unique box of $D_{j_2}$ below row $i_2$, since otherwise there would result in  a subdiagram of $D$ which is the same as the  configuration    $(A)$ (possibly by swapping the columns) in Figure \ref{avoid}. 

(3)  Suppose to the contrary that there are two    $\mathrm{Type\ III} $ columns (say, $D_{j_1}$ and $D_{j_2}$ with $j_1<j_2$)  occurring. Again, let us use  $(i_1, j_1), (i_2, j_1)$ (respectively, $(i_1, j_2), (i_2, j_2)$) to denote  positions of the first two crossings in column $j_1$ (respectively, column $j_2$).
In  $D_{j_1}$ (respectively, $D_{j_2}$), locate the   top most two boxes  below row $i_2$, denoted    $(h_1,j_1)$
 and $(h_2,j_1)$ with $i_2<h_1<h_2$ (respectively, $(\ell_1,j_2)$
 and $(\ell_2,j_2)$ with $i_2<\ell_1<\ell_2$). 
 If $h_1=\ell_1$, then the subdiagram, restricted to rows $\{i_1, i_2, h_1, h_2\}$ and columns $\{j_1, j_2\}$, is the configuration $(C)$ in Figure \ref{avoid}. If $h_1\neq \ell_1$, then the subdiagram, restricted to rows $\{i_2, h_1, \ell_1\}$ and columns $\{j_1, j_2\}$, is the configuration $(A)$  (possibly after a swapping of columns) in Figure \ref{avoid}. In both situations, we are led to a contradiction.
\end{proof}

\begin{figure}[h]
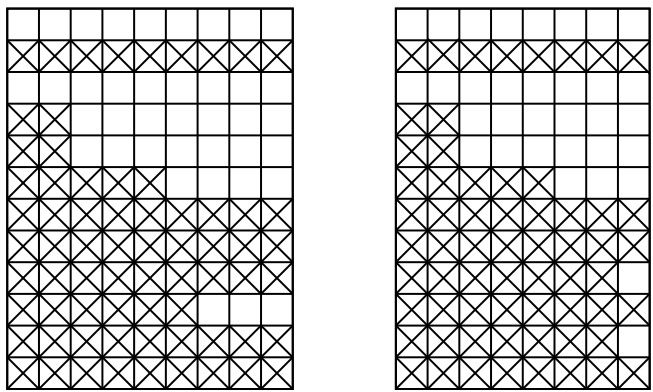

    \centering

\tikzset{every picture/.style={line width=0.75pt}} 


    \caption{Illustrations of the regions in (2) and (3) of Lemma \ref{BGTR}, respectively.}
    \label{fig:enter-label-T}
\end{figure}

\subsection{Columns in distinct regions}

We now investigate the   columns in distinct regions. Both lemmas in this subsection will be used in Subsections \ref{sub4.2-1} and \ref{subsection-4.3}.

\begin{lem}\label{lower than second last}
Let $D=(D_1,\ldots, D_n)$ be a normalized multiplicity-free diagram. Suppose that  $D_{j_1}$ and $D_{j_2}$ with $j_1<j_2$ lie in distinct regions, and the first crossing of column $j_1$ (respectively, column $j_2$) is in row $i_1$ (respectively, row $i_2$).  Suppose further that $\emptyset \neq D_{j_1}=\{d_1<d_2<\cdots<d_k\}$, and that  $D_{j_2}\neq [i_2-1]$ (that is, there is at least one box of $D_{j_2}$ lying below row $i_2$). Then   we have $i_2>d_{k-1} $ (here we set $d_0=0$). In other words, the first crossing in column $j_2$ is  below row $d_{k-1}$.
\end{lem}

Before giving a proof of this lemma, let us explain that  the assumption  that $D_{j_2}\neq [i_2-1]$ is natural.
Suppose that $D$ has a column $D_j=[m]$. Such a column is called a standard interval in \cite{mestrans}.  Then there is only one choice of $C_j$ such that $C_j\leq D_j$, namely, $C_j=D_j=[m]$.
So the diagram $D'$ obtained by removing $D_j$ from $D$ inherits all information of $D$ in the sense that $\chi_{D}(x)=x_1\cdots x_m\, \chi_{D'}(x)$.
So, to prove Theorem \ref{MMain},  we may without loss of generality   require that   $D$ has no  standard intervals. 

\begin{proof}[Proof of Lemma \ref{lower than second last}]
The assertion is trivial when $k=1$. Assume now  that $k\geq 2$.
Suppose to the contrary that  $i_2\leq d_{k-1}$. We shall deduce the contradiction that $D_{j_2}=[i_2-1]$.  Since $D_{j_1}$ and $D_{j_2}$ lie in distinct regions, we have  $i_1<i_2$.  Our aim is  to verify that any box in column   $j_2$, lying below row $i_2$, contains a crossing.  
There are two cases.

Case 1: The box $(i_2, j_1)$ belongs to $D_{j_1}$. 
In this case, we have  $d_k>i_2$. To avoid the configuration $(D)$ in Figure \ref{avoid}, the box $(d_k, j_2)$ must contain a crossing, as illustrated in Figure \ref{AASJ-12}. 
\begin{figure}[h]
        \centering
		\begin{tikzpicture}

                \draw (0.5,1) rectangle (1,0.5);
                \draw (0.5+1,1) rectangle (1+1,0.5);
                \draw (0.5,1-1) rectangle (1,0.5-1);
                \draw (0.5+1,1-1) rectangle (1+1,0.5-1);
                \draw (0.5,1-2) rectangle (1,0.5-2);
                \draw (0.5+1,1-2) rectangle (1+1,0.5-2);
                \draw[black] (0.5,1) -- (1,0.5);
                \draw[black] (1,1) -- (0.5,0.5);
               
                \draw[black] (1+0.5,0) -- (1.5+0.5,-0.5);
                \draw[black] (1+0.5,-0.5) -- (1.5+0.5,0);
                \draw[black] (1+0.5,0-1) -- (1.5+0.5,-0.5-1);
                \draw[black] (1+0.5,-0.5-1) -- (1.5+0.5,0-1);
                \node at (0,0.7) {$i_1$};
                \node at (0,0.2-0.5) {$i_2$};
                \node at (0,-0.3-1) {$d_k$};
                \node at (0.75, -1.3+0.5-1+3.2) {$j_1$};
                \node at (1.25+0.5, -1.3+0.5-1+3.2) {$j_2$};
		\end{tikzpicture}
        \caption{An illustration for the proof of Case 1. }\label{AASJ-12}
  
    \end{figure}
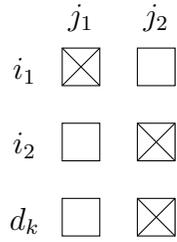
Consider the box $(i, j_2)$   with $i>d_k$. Since each box in column $j_1$, lying  below row $d_k$, contains a crossing, to avoid the configuration $(E)$ in Figure \ref{avoid}, the box $(i, j_2)$ must contain a crossing. 

It remains to  consider the box $(r, j_2)$ with $i_2<r<d_k$. 
We assert that  $(r, j_2)\not\in D_{j_2}$.
Suppose otherwise that 
$(r, j_2)\in D_{j_2}$.
If $(r, j_1)$ belongs to $D_{j_1}$, then the boxes of $D$ in rows $\{i_1, i_2, r\}$ and columns $\{j_1, j_2\}$ form the configuration $(D)$ in Figure \ref{avoid}, and if $(r, j_1)$ does not belong to $D_{j_1}$, then the boxes of $D$ in rows $\{i_1, i_2, r, d_p\}$ and columns $\{j_1, j_2\}$   form the  configuration by swapping the columns of   $(F)$  in Figure \ref{avoid}. In either  case, we are led to a contradiction. This concludes  the assertion.

Case 2: The box $(i_2, j_1)$ contains  a crossing.  In this case, we have $d_{k-1}>i_2$.   To avoid the configuration $(B)$ in Figure \ref{avoid}, the box $(d_{k-1}, j_2)$ must contain a crossing. Moreover, to avoid the configuration $(A)$ in Figure \ref{avoid},  any box in column $j_2$, which lies below row $d_{k-1}$, contains a crossing. The boxes in rows $\{i_1, i_2, d_{k-1}, d_k\}$ and columns $\{j_1, j_2\}$ look as depicted in Figure \ref{avoid BB}.  
\begin{figure}[h]
        \centering
		\begin{tikzpicture}

                \draw (-3.5,1) rectangle (-3,0.5);
                \draw (-3.5+1,1) rectangle (-3+1,0.5);
                \draw (-3.5,1-1) rectangle (-3,0.5-1);
                \draw (-3.5+1,1-1) rectangle (-3+1,0.5-1);
                \draw (-3.5,1-2) rectangle (-3,0.5-2);
                \draw (-3.5+1,1-2) rectangle (-3+1,0.5-2);
                \draw (-3.5,1-2-1) rectangle (-3,0.5-2-1);
                \draw (-3.5+1,1-2-1) rectangle (-3+1,0.5-2-1);
                \draw[black] (-3,1) -- (-3.5,0.5);
                \draw[black] (-3,0.5) -- (-3.5,1);
                \draw[black] (-3,1-1) -- (-3.5,0.5-1);
                \draw[black] (-3,0.5-1) -- (-3.5,1-1);
                \draw[black] (-3+1,1-1) -- (-3.5+1,0.5-1);
                \draw[black] (-3+1,0.5-1) -- (-3.5+1,1-1);
                \draw[black] (-3+1,1-1-1) -- (-3.5+1,0.5-1-1);
                \draw[black] (-3+1,0.5-1-1) -- (-3.5+1,1-1-1);
                \draw[black] (-3+1,1-1-1-1) -- (-3.5+1,0.5-1-1-1);
                \draw[black] (-3+1,0.5-1-1-1) -- (-3.5+1,1-1-1-1);
                \node at (-3.25, -1.3+0.5-1+3.2) {$j_1$};
                \node at (-2.25, -1.3+0.5-1+3.2) {$j_2$};
                \node at (-4,0.7) {$i_1$};
                \node at (-4,0.2-0.5) {$i_2$};
                \node at (-4,-0.3-1) {$d_{k-1}$};
                \node at (-4,-0.3-2) {$d_k$};
               
 		\end{tikzpicture}
   \caption{An illustration for the proof of Case 2.}
        \label{avoid BB}
    \end{figure}
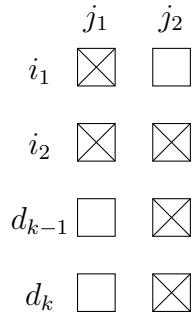

We still need to check that any box in column $j_2$, which lies between row $i_2$ and row $d_{k-1}$, contains a crossing. Suppose otherwise that there exists a box $(i, j_2)$ with $i_2<i<d_{k-1}$ belonging to $D_{j_2}$. Consider the box $(i, j_1)$. If it belongs to $D_{j_1}$, then the boxes in   rows $\{i_1,i_2, i, d_{k-1}\}$ and columns $\{j_1, j_2\}$ would form the configuration $(B)$ in Figure \ref{avoid}, and otherwise, the boxes in  rows $\{i_2, i, d_{k-1}\}$ and columns $\{j_1, j_2\}$ would form the configuration by swapping the columns of $(A)$   in Figure \ref{avoid}.  This verifies  that any box $(i, j_2)$ with $i_1<i<d_{p-1}$ contains a crossing. 
So the proof is complete. 
\end{proof}

\begin{lem}\label{only one}
Take the same assumptions as in Lemma \ref{lower than second last} (so $i_2>d_{k-1}$). 
\begin{itemize}
    \item[(1)] If  
$i_2<d_k-1$, then $D_{j_2}=[i_2-1]\cup \{d_k\}$.

    \item[(2)] If  
$i_2=d_k-1$, then $[i_2-1]\cup \{d_k\} \subseteq D_{j_2}$.
\end{itemize}

\end{lem}

\begin{proof}
Keep in mind that $i_2>d_{k-1}$. 

(1) First, to avoid the configuration obtained  by swapping the columns of $(A)$   in Figure \ref{avoid}, the box  $(d_k-1, j_2)$ must contain  a crossing. So the boxes in rows $\{d_{k-1}, i_2, d_k-1, d_k\}$ and columns $\{j_1, j_2\}$ are as illustrated in Figure \ref{RGI-21}.
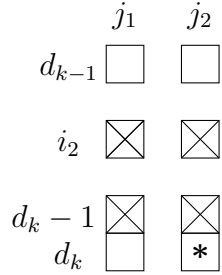
\begin{figure}[ht]
        \centering
		\begin{tikzpicture}

                \draw (-3.5,1) rectangle (-3,0.5);
                \draw (-3.5+1,1) rectangle (-3+1,0.5);
                \draw (-3.5,1-1) rectangle (-3,0.5-1);
                \draw (-3.5+1,1-1) rectangle (-3+1,0.5-1);
                \draw (-3.5,1-2) rectangle (-3,0.5-2);
                \draw (-3.5+1,1-2) rectangle (-3+1,0.5-2);
                \draw (-3.5,1-2-1+0.5) rectangle (-3,0.5-2-1+0.5);
                \draw (-3.5+1,1-2-1+0.5) rectangle (-3+1,0.5-2-1+0.5);
                \draw[black] (-3,1-1) -- (-3.5,0.5-1);
                \draw[black] (-3,0.5-1) -- (-3.5,1-1);
                \draw[black] (-3,1-1) -- (-3.5,0.5-1);
                \draw[black] (-3,0.5-1) -- (-3.5,1-1);
                \draw[black] (-3+1,1-1) -- (-3.5+1,0.5-1);
                \draw[black] (-3+1,0.5-1) -- (-3.5+1,1-1);
                \draw[black] (-3+1,1-1-1) -- (-3.5+1,0.5-1-1);
                \draw[black] (-3+1,0.5-1-1) -- (-3.5+1,1-1-1);
                \draw[black] (-3,1-1-1) -- (-3.5,0.5-1-1);
                \draw[black] (-3,0.5-1-1) -- (-3.5,1-1-1);
                \node at (-3.25, -1.3+0.5-1+3.2) {$j_1$};
                \node at (-2.25, -1.3+0.5-1+3.2) {$j_2$};
                \node at (-4,0.7) {$d_{k-1}$};
                \node at (-4,0.2-0.5) {$i_2$};
                \node at (-4.2,-0.3-1) {$d_{k}-1$};
                \node at (-4,-0.3-2+0.5) {$d_k$};
                \node at (-2.25,-1.75) {$\boldsymbol{*}$};
               
 		\end{tikzpicture}
   \caption{An illustration for the proof of (1).}\label{RGI-21}
    \end{figure}

For the same reason as above, any box in column $j_2$, lying between row $i_2$ and row $d_k-1$, contains a crossing. Let us look at the box $(d_k, j_2)$. We assert that this box belongs to $D_{j_2}$. Suppose otherwise that $(d_k, j_2)$ contains a crossing. To avoid the configuration by swapping the columns of $(A)$  in Figure \ref{avoid}, any box in column $j_2$ lying below row $d_k$ contains a crossing. This, along with the above analysis,  implies that $D_{j_2}=[i_2-1]$, contradicting  the assumption that $D_{j_2}\neq [i_2-1]$. So    the assertion is verified.  

Now, because $(d_k, j_2)\in D_{j_2}$, any box $(i, j_2)$ with $i>d_k$ contains a crossing since otherwise the boxes in rows $\{i_2, d_k-1, d_k, i\}$ and columns $\{j_1, j_2\}$ would yield the  configuration by swapping the columns of $(C)$  in Figure \ref{avoid}. This shows that $D_{j_2}=[i_2-1]\cup \{d_k\}$.

(2) We need to check  that $(d_k, j_2)\in D_{j_2}$. Suppose otherwise that 
$(d_k, j_2)\not\in D_{j_2}$. Since $D_{j_2}\neq [i_2-1]$, there exists a box $(i, j_2)$ with $i>d_k$ belonging to $D_{j_2}$. Then the subdiagram, restricted to rows $\{d_k-1, d_k, i\}$ and columns $\{j_1, j_2\}$, becomes   the configuration $(A)$ in Figure \ref{avoid}, leading to a contradiction.  
\end{proof}

\section{Proof of Theorem \ref{MMain}}\label{MMain-pp}

To complete the proof of Theorem \ref{MMain}, by Proposition \ref{PPPo}, it is equivalent to showing that when $D=(D_1,\ldots, D_n)$ is multiplicity-free, the eigenspace 
\[
\mathrm{Span}_{\mathbb{C}}\left\{\mathrm{det}\left(Y_{D}^C\right)\colon  C\leq D, \, x^C=x^a\right\}
\] 
has dimension   one. 
We   do this by showing  that $\mathrm{det}\left(Y_{D}^C\right)=\mathrm{det}\left(Y_{D}^{C'}\right)$ for  $C, C'\leq D$ with $x^C=x^{C'}$.  In light of Proposition \ref{ABC-1}, this will be  achieved by establishing a bijection   between the sets $\mathcal{F}_{D}(C)$ and $\mathcal{F}_{D}(C')$ which preserves both the sign and the weight.

\begin{theo}\label{main-bijection}
Let $D=(D_1,\ldots, D_n)$ be a normalized multiplicity-free diagram, and let $C=(C_1,\ldots, C_n) $ and $ C'=(C_1',\ldots, C'_n)$ be two  distinct  diagrams,  that are less than or equal to $D$, such that $x^C=x^{C'}$. Then there is a sign- and weight-preserving   bijection between the sets $\mathcal{F}_{D}(C) $  and  $\mathcal{F}_{D}(C')$.
\end{theo}

Without loss of generality, we may  make the following extra  assumptions  on $D$: 
\begin{itemize} 
    \item[] (C 1): each column in $[n]\times [n]$ has at least two crossings.

\item[] (C 2): $D$ has no column of the form $[m]$ where $m\geq 1$; 
\end{itemize}
The reasons for (C 1) and (C 2) have been  explained respectively  in the beginning of
Subsection \ref{subss} and  below Lemma \ref{lower than second last}.

The construct of the bijection  is an iterative procedure, based on an operation, denoted $\Phi$, acting on flagged fillings. 
The operation depends only on $D, C$ and $C'$.
Roughly speaking, for $F\in \mathcal{F}_{D}(C)$, the operation  $\Phi$ interchanges the columns of $F$ by sliding certain entries  along the same rows. 
After applying   $\Phi$,
we get a new flagged filling $\Phi(F)$,  which is ``closer'' to a flagged filling in $\mathcal{F}_{D}(C')$ in the sense that there are more columns whose  entries match  the subsets in $C'$. Then we replace $F$ by a flagged filling, denoted $\hat{\Phi}(F)$, which is obtained from $\Phi(F)$ by ignoring certain columns that have been  adjusted  by $\Phi$, and we continue  to apply   $\Phi$ to $\hat{\Phi}(F)$. Eventually, we arrive at a flagged filling belonging to  $\mathcal{F}_{D}(C')$, which is defined as the image of $F$.

The description of $\Phi$ will be distinguished into three cases, depending on the configuration of the first region of $D$. 
Specifically, according to Lemma \ref{BGTR},  each region of $D$ is one of the following three types: it contains
\begin{itemize} 
    \item[] (R 1):   only Type I columns, or
    
    \item[] (R 2):  Type I columns followed by some  copies of a Type II column, or
    
    \item[]  (R 3): type I columns followed by exactly one  Type III column.
\end{itemize}
In the remaining of this section,  let $F=(F_1,\ldots, F_n)$ be any fixed filling in $\mathcal{F}_D(C)$. Our task is to construct    $\Phi(F)$, as well as $\hat{\Phi}(F)$ obtained from $\Phi(F)$ by ignoring some specific columns. 

\subsection{The first region of $D$ is of Type (R 1)}\label{subset4.1}

In this case, the first region of $D$   contains only  Type I columns.
Suppose that there are $m$ columns in the first region. We may further suppose that  for $1\leq j\leq m$, $D_j$ is not empty (since empty columns can be obviously ignored). Recall that $D$ satisfies the   requirement (C 2) given below Theorem \ref{main-bijection}.
Therefore, for $1\leq j\leq m$, the signature of $D_j$  is positive. As an illustration,  Figure \ref{IRegion-1} lists the columns in the first region. Here we have erased  some rows, lying  on the bottom, which contain only crossings.   
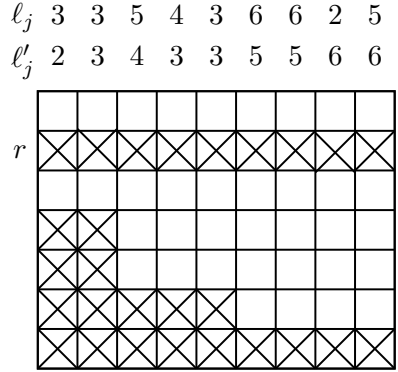
\begin{figure}[h]
    \centering
\addtocounter{MaxMatrixCols}{10}

\tikzset{every picture/.style={line width=0.75pt}} 

\begin{tikzpicture}[x=0.75pt,y=0.75pt,yscale=-1,xscale=1]

\draw  [draw opacity=0] (52,68) -- (232,68) -- (232,208) -- (52,208) -- cycle ; \draw   (72,68) -- (72,208)(92,68) -- (92,208)(112,68) -- (112,208)(132,68) -- (132,208)(152,68) -- (152,208)(172,68) -- (172,208)(192,68) -- (192,208)(212,68) -- (212,208) ; \draw   (52,88) -- (232,88)(52,108) -- (232,108)(52,128) -- (232,128)(52,148) -- (232,148)(52,168) -- (232,168)(52,188) -- (232,188) ; \draw   (52,68) -- (232,68) -- (232,208) -- (52,208) -- cycle ;
\draw    (52,88) -- (72,108) ;
\draw    (52,108) -- (72,88) ;

\draw    (72,87) -- (92,107) ;
\draw    (72,107) -- (92,87) ;

\draw    (92,88) -- (112,108) ;
\draw    (92,108) -- (112,88) ;

\draw    (112,88) -- (132,108) ;
\draw    (112,108) -- (132,88) ;

\draw    (132,88) -- (152,108) ;
\draw    (132,108) -- (152,88) ;

\draw    (152,88) -- (172,108) ;
\draw    (152,108) -- (172,88) ;

\draw    (172,88) -- (192,108) ;
\draw    (172,108) -- (192,88) ;

\draw    (192,89) -- (212,109) ;
\draw    (192,109) -- (212,89) ;

\draw    (212,88) -- (232,108) ;
\draw    (212,108) -- (232,88) ;
\draw    (52,128) -- (72,148) ;
\draw    (52,148) -- (72,128) ;

\draw    (72,128) -- (92,148) ;
\draw    (72,148) -- (92,128) ;

\draw    (152,188) -- (172,208) ;
\draw    (152,208) -- (172,188) ;

\draw    (172,188) -- (192,208) ;
\draw    (172,208) -- (192,188) ;

\draw    (192,188) -- (212,208) ;
\draw    (192,208) -- (212,188) ;

\draw    (212,188) -- (232,208) ;
\draw    (212,208) -- (232,188) ;

\draw    (132,168) -- (152,188) ;
\draw    (132,188) -- (152,168) ;

\draw    (132,188) -- (152,208) ;
\draw    (132,208) -- (152,188) ;

\draw    (112,168) -- (132,188) ;
\draw    (112,188) -- (132,168) ;

\draw    (112,188) -- (132,208) ;
\draw    (112,208) -- (132,188) ;

\draw    (92,168) -- (112,188) ;
\draw    (92,188) -- (112,168) ;

\draw    (92,188) -- (112,208) ;
\draw    (92,208) -- (112,188) ;

\draw    (72,168) -- (92,188) ;
\draw    (72,188) -- (92,168) ;

\draw    (72,188) -- (92,208) ;
\draw    (72,208) -- (92,188) ;

\draw    (52,168) -- (72,188) ;
\draw    (52,188) -- (72,168) ;

\draw    (52,188) -- (72,208) ;
\draw    (52,208) -- (72,188) ;

\draw    (72,148) -- (92,168) ;
\draw    (72,168) -- (92,148) ;

\draw    (52,148) -- (72,168) ;
\draw    (52,168) -- (72,148) ;

\draw (-1+38,43) node [anchor=north west][inner sep=0.75pt]  [font=\small] [align=left] {$\ell'_j$};
\draw (-1+58,44) node [anchor=north west][inner sep=0.75pt]  [font=\small] [align=left] {{2}};
\draw (-1+78,44) node [anchor=north west][inner sep=0.75pt]  [font=\small] [align=left] {{3}};
\draw (-1+98,44) node [anchor=north west][inner sep=0.75pt]  [font=\small] [align=left] {{4}};
\draw (-1+118,44) node [anchor=north west][inner sep=0.75pt]  [font=\small] [align=left] {{3}};
\draw (-1+138,44) node [anchor=north west][inner sep=0.75pt]  [font=\small] [align=left] {{3}};
\draw (-1+158,44) node [anchor=north west][inner sep=0.75pt]  [font=\small] [align=left] {{5}};
\draw (-1+178,44) node [anchor=north west][inner sep=0.75pt]  [font=\small] [align=left] {{5}};
\draw (-1+198,44) node [anchor=north west][inner sep=0.75pt]  [font=\small] [align=left] {{6}};
\draw (-1+218,44) node [anchor=north west][inner sep=0.75pt]  [font=\small] [align=left] {{6}};
\draw (-1+38,22) node [anchor=north west][inner sep=0.75pt]  [font=\small] [align=left] {$\ell_j$};
\draw (-1+58,23) node [anchor=north west][inner sep=0.75pt]  [font=\small] [align=left] {{3}};
\draw (-1+78,23) node [anchor=north west][inner sep=0.75pt]  [font=\small] [align=left] {{3}};
\draw (-1+98,23) node [anchor=north west][inner sep=0.75pt]  [font=\small] [align=left] {{5}};
\draw (-1+118,23) node [anchor=north west][inner sep=0.75pt]  [font=\small] [align=left] {{4}};
\draw (-1+138,23) node [anchor=north west][inner sep=0.75pt]  [font=\small] [align=left] {{3}};
\draw (-1+158,23) node [anchor=north west][inner sep=0.75pt]  [font=\small] [align=left] {{6}};
\draw (-1+178,23) node [anchor=north west][inner sep=0.75pt]  [font=\small] [align=left] {{6}};
\draw (-1+198,23) node [anchor=north west][inner sep=0.75pt]  [font=\small] [align=left] {{2}};
\draw (-1+218,23) node [anchor=north west][inner sep=0.75pt]  [font=\small] [align=left] {{5}};
\draw (38,94) node [anchor=north west][inner sep=0.75pt]  [font=\small] [align=left] {$r$};

\end{tikzpicture}
\caption{Columns in a Type (R 1) region and their labels.}\label{IRegion-1}
\end{figure}

To describe the operation $\Phi$, we  label the   $m$ columns   in two ways.  
Suppose that the first crossing in each column is in  row $r$. Let $1\leq j\leq m$.  Assume that  $(n_j, j)$ is the box right above the second crossing in column $j$ (which is the lowest box of $D_j$). Note that $n_1\leq n_2\leq \cdots\leq n_m$. For the region in Figure \ref{IRegion-1}, we have $m=9$ and 
$(n_1,\ldots, n_9)=(3,3,5,5,5,6,6,6,6)$. 
Observe that $C_j=[{n}_j]\setminus \{\ell_j\}$  for some $r\leq  \ell_j \leq {n}_j$, and $C_j'=[{n}_j]\setminus \{\ell_j'\}$ for some  $r\leq  \ell_j' \leq {n}_j$. 
By assigning    $D_j$  with a label {$\ell_j$} or {$\ell_j'$}, we obtain two kinds of labelings for the $m$ columns.  See Figure \ref{IRegion-1} for an illustration of the labelings,   from which we could  recover
$C_j$ and $C'_j$.

\begin{lem}\label{milti-4.1}
As multisets, $\{\ell_1,\ldots, \ell_m\}$  is the same as     $\{\ell_1',\ldots, \ell_m'\}$.
\end{lem}

\begin{proof}
For notational simplicity, we denote $p=n_m$.  Write 
\[
x^C=x^{C'}=x_1^{a_1}\cdots x_p^{a_p}\,x_{p+1}^{a_{p+1}}\cdots x_n^{a_n}.
\]
Consider the factor $x_1^{a_1}\cdots x_p^{a_p}$. It can be seen  that \begin{align*}
x_1^{a_1}\cdots x_p^{a_p}&=\frac{\prod_{j=1}^{m}x_1\cdots x_{n_j}}{x_{\ell_1}\cdots x_{\ell_m}}\prod_{j={m+1}}^n\prod_{t\in C_j \atop  1\leq t\leq p} x_t \\
&=\frac{\prod_{j=1}^{m}x_1\cdots x_{n_j}}{x_{\ell_1'}\cdots x_{\ell_m'}}\prod_{j={m+1}}^n\prod_{t\in C_j' \atop  1\leq t\leq p} x_t.
\end{align*}
By Lemma \ref{lower than second last}, for $m+1\leq j\leq n$, the first crossing in column $j$ is strictly lower than row $p-1$, and thus we have $[p-1]\subseteq D_j$, and hence $[p-1]\subseteq C_j$ and $[p-1]\subseteq C_j'$.
Therefore,
\begin{align*}
x_1^{a_1}\cdots x_p^{a_p}&=\frac{\prod_{j=1}^{m}x_1\cdots x_{n_j}}{x_{\ell_1}\cdots x_{\ell_m}}(x_1\cdots x_{p-1})^{n-m}\prod_{j={m+1}}^n\prod_{p\in C_j} x_p \\
&=\frac{\prod_{j=1}^{m}x_1\cdots x_{n_j}}{x_{\ell_1'}\cdots x_{\ell_m'}}(x_1\cdots x_{p-1})^{n-m}\prod_{j={m+1}}^n\prod_{p\in C_j'} x_p.
\end{align*}
From the above, we know  that $\{m+1\leq j\leq n\colon p\in C_j\}$ and $\{m+1\leq j\leq n\colon p\in C_j'\}$ have the same cardinality. So we have $x_{\ell_1}\cdots x_{\ell_m}=x_{\ell_1'}\cdots x_{\ell_m'}$, which justifies   $\{\ell_1,\ldots, \ell_m\}=\{\ell_1',\ldots, \ell_m'\}$ as multisets. 
\end{proof}

The operation $\Phi$  will be performed on the columns of $F$ in the first region, based on a column-interchanging procedure. Let  $1\leq j_{1}<j_{2}\leq m$. Note that if $(i, j_1)$ is  a box of $F_{j_1}$, then $(i, j_2)$ is  a box of $F_{j_2}$. 
We say that  $F_{j_1}$ and $F_{j_2}$ are {\it interchangeable} if the resulting filling  by swapping the entries   in the same rows of $F_{j_1}$ and $F_{j_2}$ (leaving the entries of $F_{j_2}$ below row $n_{j_1}$ unchanged)  is still a  flagged filling. Equivalently, the entries of $F_{j_1}$ have no intersetion with the entries of $F_{j_2}$ that lie    below row $n_{j_1}$.

\begin{lem}\label{lem-4.3pj}
If $\ell_{j_2}\leq n_{j_1}$, then $F_{j_1}$ and $F_{j_2}$ are interchangeable.

\end{lem}

\begin{proof}
Recall that for $1\leq j\leq m$, $C_j=[n_j]\setminus \{\ell_j\}$.
When $\ell_{j_2}\leq n_{j_1}$,   the elements in the interval  $[n_{j_1}+1, n_{j_2}]$ all belong to $C_{j_2}$, and they occupy the boxes of $F_{j_2}$ below row $n_{j_1}$. So, after the column-interchanging,  the resulting  filling is still  a flagged filling. 
\end{proof}

We can now describe the  construction of  $\Phi(F)$ in the Type (R 1) case.

Step 1. Denote $c=\max\{\ell_1,\ldots, \ell_m\}$. Let $A=\{1\leq j\leq m\colon \ell_j=c\}$,  $B=\{1\leq j\leq m\colon \ell_j'=c\}$ and $S=A\cap B$. 
By Lemma  \ref{milti-4.1}, we have $|A|=|B|$.
Assume that $A\setminus S=\{j_1<\cdots<j_t\}$ and $B\setminus S=\{j_1'<\cdots<j_t'\}$. 
For every $1\leq s\leq t$, swap the entries in the same rows of $F_{j_s}$ and $F_{j_s'}$ (clearly, this is independent of the order of $s$).

We  explain that   after Step 1, the resulting filling is also  a flagged filling. Consider  two cases: $j_s>j_s'$ and $j_s<j_s'$. In the former case, we have $\ell_{j_s}=c=\ell'_{j'_{s}}\leq n_{j'_{s}}$, and in the latter  case, we have $\ell_{j_s'}<c=\ell_{j_s}\leq n_{j_s}$. Both cases satisfy the requirement in 
 Lemma \ref{lem-4.3pj}, and thus $F_{j_s}$ and $F_{j_s'}$  are interchangeable. 

Step 2. After Step 1,  
the columns  labeled with $c$ are indexed by the set $B$. 
Ignore such columns, and then return to   Step 1 where the operation  is applied to  the remaining columns  in the first region. 
 
The above algorithm eventually  terminates. Define $\Phi(F)$ to be the resulting flagged filling.  Figure \ref{GJO-2} gives an illustration of  the construction of $\Phi$.
Note that for  $1\leq j\leq m$, the label of the $j$-th column in $\Phi(F)$  is equal to $\ell_j'$.

\begin{figure}[h]
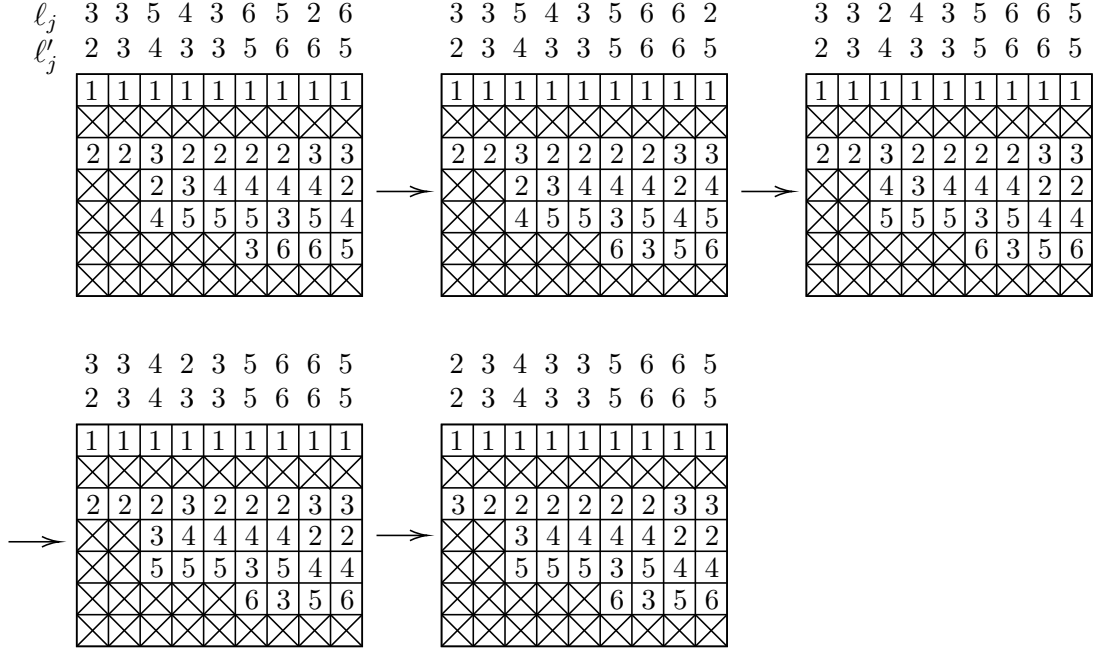

    \centering

\tikzset{every picture/.style={line width=0.6pt}} 


   \caption{An illustration of the construction of $\Phi$ in the Type (R 1) case.}\label{GJO-2}
\end{figure}

Assume that $\Phi(F)=(F'_1,\ldots, F_n')$. For $1\leq j\leq m$, the set of entries in $F_j'$ is equal to  $C_j'$, while for $m<j\leq n$, we have $F_j'=F_j$. This means that  $\Phi(F)$ is ``closer'' to a filling in $\mathcal{F}_D(C')$. 
We set $\hat{\Phi}(F)$ to be the flagged filling obtained from $\Phi(F)$ by ignoring the columns in the first region. 

\subsection{The first region of $D$ is of Type (R 2)}\label{sub4.2-1}

Still,  let $(D_1,\ldots, D_m)$ be the columns of $D$ in the first region. As in Subsection \ref{subset4.1}, we assume that each $D_j$ for $1\leq j\leq m$ is not empty. We also adopt the notation $n_j$ which means that $(n_j,j)$ is the box right above the second crossing in  column $j$. Note  that  when $D_j$ is a column of Type II,   $(n_j,j)$ is not the lowest box of $D_j$. 

Imitating Subsection \ref{subset4.1}, we label each column $D_j$ for $1\leq j\leq m$ in two ways, also denoted by $\ell_j$ and $\ell_j'$. Let us first define $\ell_j$. 
Consider $[n_j]\setminus C_j$. If $D_j$ is of Type I, then $[n_j]\setminus C_j$ contains exactly a single element, just as  in Subsection \ref{subset4.1}. While, if $D_j$ is of Type II, then $[n_j]\setminus C_j$ either contains a single element or is the empty set. Clearly, $[n_j]\setminus C_j=\emptyset$ if and only if $C_j=[n_j]$. When $[n_j]\setminus C_j$   contains a single element,  let $\ell_j$ be this element, and when $[n_j]\setminus C_j=\emptyset$,  let $\ell_j=\emptyset$.   In completely the same manner, we  may define $\ell'_j$   with respect to  $C'_j$. Figure  \ref{YTRO-23} illustrates the columns, as well as their labels, in the first (Type (R 2)) region of a filling $F$ in $\mathcal{F}_{D}(C)$. 
\begin{figure}[h]
    \centering
    \tikzset{every picture/.style={line width=0.5pt}} 


\caption{An illustration of columns  and their labels in a Type (R 2) region.}
    \label{YTRO-23}
\end{figure}

\begin{lem}\label{lem-4.4-1}
    As multisets, 
    $\{\ell_1,\ldots, \ell_m\}$ is the same as $\{\ell_1',\ldots, \ell_m'\}$.
\end{lem}

\begin{proof}
 The proof is similar to (actually slightly easier than) that of Lemma  \ref{milti-4.1}.  Let $p=n_m$, and write 
$x^C=x^{C'}=x_1^{a_1}\cdots x_p^{a_p}\,x_{p+1}^{a_{p+1}}\cdots x_n^{a_n}.$ By Lemma \ref{lower than second last}, for $m+1\leq j\leq n$, the first crossing in column $j$ is lower than row $p$, which implies   $[p]\subseteq D_j$ and in turn  $[p]\subseteq C_j$ and $[p]\subseteq C_j'$.
Now we see that
\begin{align*}
x_1^{a_1}\cdots x_p^{a_p}=\frac{\prod_{j=1}^{m}x_1\cdots x_{n_j}}{x_{\ell_1}\cdots x_{\ell_m}}(x_1\cdots x_{p})^{n-m}
=\frac{\prod_{j=1}^{m}x_1\cdots x_{n_j}}{x_{\ell_1'}\cdots x_{\ell_m'}}(x_1\cdots x_{p})^{n-m},
\end{align*}
where we set $x_{\emptyset}=1$.
It follows  that $x_{\ell_1}\cdots x_{\ell_m}=x_{\ell_1'}\cdots x_{\ell_m'}$, and so  $\{\ell_1,\ldots, \ell_m\}=\{\ell_1',\ldots, \ell_m'\}$ as multisets.
\end{proof}

Unlike  Subsection \ref{subset4.1}, the construction of $\Phi(F)$ in the Type (R 2) case requires two algorithms. The  first algorithm acts on the columns in the first region   in the  same way as the algorithm in Subsection \ref{subset4.1}.   Here  the label     $\emptyset$ is regarded  as the infinity  which is  greater than any positive integer.  
This procedure   will be best understood via an example in Figure \ref{fig:enter-label1}. 
 
\begin{figure}[!h]
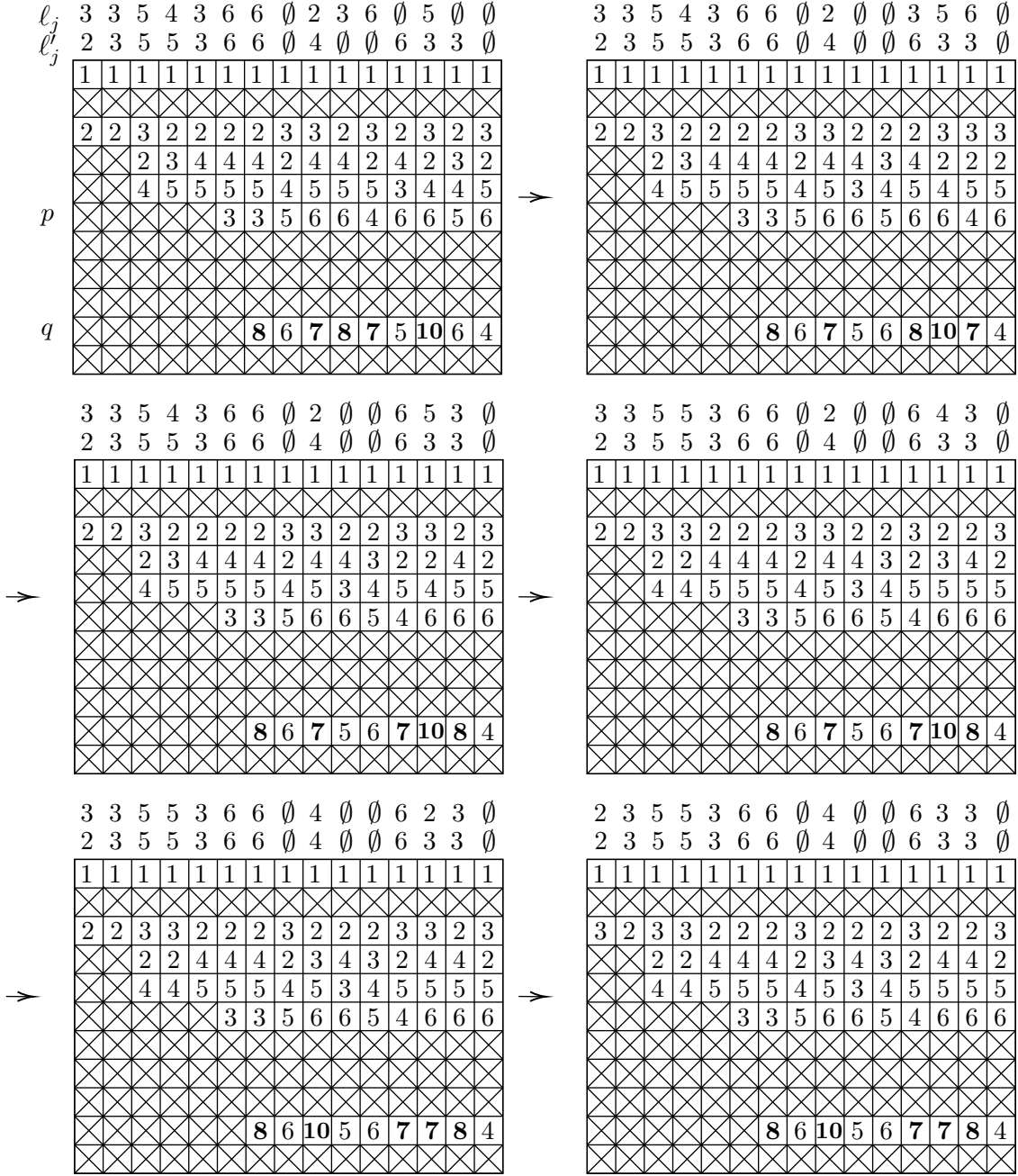

    \centering
    
\tikzset{every picture/.style={line width=0.5pt}} 


    \caption{Illustration of the first algorithm in the Type (R 2) case.}
    \label{fig:enter-label1}
\end{figure}

After applying  the above algorithm, the label of $D_j$ for $1\leq j\leq m$ becomes $\ell_j'$.
Denote by $F^{(1)}$ the resulting flagged filling. Suppose that $F^{(1)}=(F^{(1)}_1,\ldots, F^{(1)}_n)$ belongs  to $\mathcal{F}_{D}(C^{(1)})$ where $C^{(1)}=(C^{(1)}_1,\ldots, C^{(1)}_n)$.  Clearly, we have $F^{(1)}_j=F_j$ for $m< j\leq n$. For  $1\leq j\leq m$,  we have $C_j^{(1)}=C_j'$ when $D_j$ is either a Type I column  or  a Type II column  labeled with $\emptyset$. 

We still need to consider the type II columns in the frist region whose  labels are not $\emptyset$. Let  $p=n_m$. By (2) in Lemma \ref{BGTR},   the lowest boxes in the Type II columns  lie in the same row, say row $q$. Our next task is to adjust the entries in row $q$ that are greater than $p$. Note that  such entries lie in  the Type II columns whose labels are not $\emptyset$. In Figure \ref{fig:enter-label1}, the entries greater than $p$ lying in row $q$ are displayed   in boldface, where $p=6$ and $q=10$. 

To deal with the entries in row $q$ that are greater than $p$, we need additionally to invoke the regions (after the first region)   whose first crossings lying in or above row $q-1$. According to  Lemma \ref{lower than second last}, the first crossing in each such region is  below row $p$.
Suppose that there are   $k-1$ ($k\geq 1$) such  regions. By Lemma \ref{only one}, each column in these regions contains a box of $D$ in row $q$. Together with the first region, we next  consider  the first $k$ regions of $D$. Assume that the first $k$ regions have a total of  $m'$ columns. 
For $1\leq r\leq k$, assume that the first crossing in the $r$-th region  is in row $i_r$.  Note that $1\leq i_1<\cdots<i_k\leq q-1$

The description of the second algorithm will be divided into     two cases, according to the relative values of $i_k$ and $q-1$. 

Case 1. $i_k<q-1$.
By (1) in Lemma \ref{only one}, for $m<j\leq m'$,   $D_j$ has exactly one box 
below the first  crossing which lies in row $q$. The second algorithm acts on the entries   in the   row $q$  of $F^{(1)}$, restricted to  the first $k$ regions, that are greater than $p$.  
Suppose that there are $d$ such entries, say $(a_1,\ldots, a_d)$ listed from left to right, and  that for $1\leq t\leq d$, $a_t$ lies in column $j_t$. An illustration  is given in  Figure \ref{alg2}, where we set  $k=3$ and $d=10$, and the entries in $(a_1,\ldots, a_{10})$ are drawn in boldface.
\begin{figure}[h]
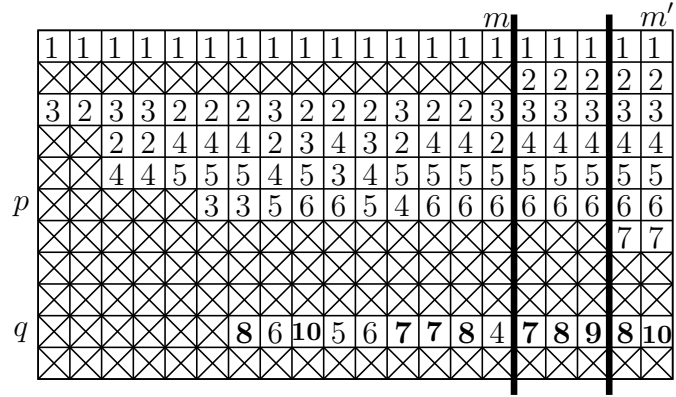

 \centering
\tikzset{every picture/.style={line width=0.6pt}} 


\caption{An illustration for   Case 1 in the second algorithm.}
\label{alg2}
\end{figure}

Notice that for $1\leq t\leq d$,  $a_t$ can be obviously  read off from $C^{(1)}_{j_t}$.
To be precise, suppose that  $D_{j_t}$ is in the $r$-th region. For $r=1$, $a_t$ is the unique element in $C^{(1)}_{j_t}$ that is greater than $p$, while for $2\leq r\leq k$, $a_t$ is the unique element in $C^{(1)}_{j_t}=C_{j_t}$ that is greater than or equal to $i_r$. We  define  $b_t$ by replacing  $C^{(1)}_{j_t}$ with $C'_{j_t}$.

\begin{lem}
As multisets, we have  $\{a_1,\ldots, a_d\}=\{b_1,\ldots, b_d\}$. 
\end{lem} 

\begin{proof}
Note that the first crossing in any column of $D$ after the $k$-th region lie in or  below   row $q$. The proof is then done by using  completely similar arguments to  those for Lemma \ref{milti-4.1}.
\end{proof}

Define $\Phi(F)$ to be the flagged filling obtained from  $F^{(1)}$ by replacing  the entry $a_t$  with $b_t$ for  $1\leq t\leq d$. Write $\Phi(F)=(F_1',\ldots, F_n')$. Then, for $1\leq j\leq m'$, the set of entries in $F_j'$ is equal to $C_j'$, while for $m'<j\leq n$, we have $F_j'=F_j$. Let $\hat{\Phi}(F)$ be obtained from $\Phi(F)$ by ignoring the first $m'$ columns.

Case 2. $i_k=q-1$.  
By (2) in Lemma \ref{only one},   each column of $D$ in the $k$-th region has  a box 
 in row $q$. 
Focus on the columns  of $F^{(1)}$ in the first $k-1$ regions (columns in the $k$-th region will be kept unchanged). We consider the entries in the $q$-th row. In the same way as in Case 1, we may define the sequences  $(a_1,\ldots, a_d)$ and $(b_1,\ldots, b_d)$.  See Figure \ref{okp-1} for an illustration, where $k=4$, and the entries  in the sequence $(a_1,\ldots, a_d)$ are shown in boldface. 
\begin{figure}[h]
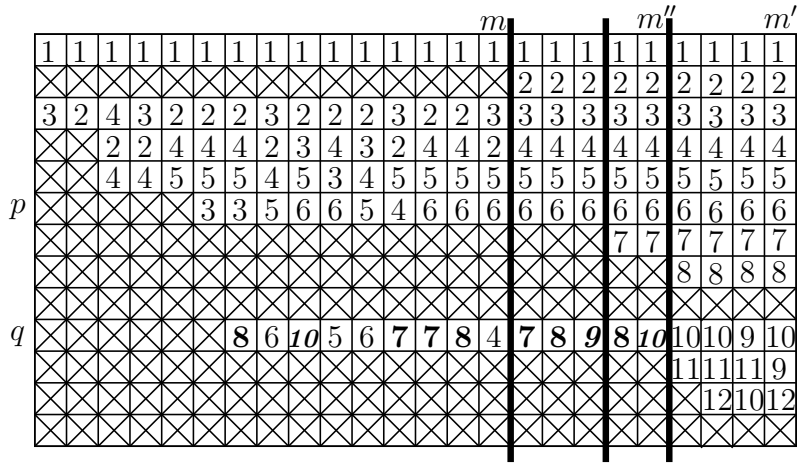

    \centering

\tikzset{every picture/.style={line width=0.6pt}} 


\caption{An illustration for Case 2 in the second algorithm.}
\label{okp-1}

\end{figure}
In this figure, we use italics to distinguish the entries equal to $q-1$ or $q$, and the reason for this will be clear later.

\begin{lem}\label{huah-1}
The multiset obtained from   $\{a_1,\ldots, a_d\}$ by removing the elements equal to $q-1$ or $q$  is the same as 
the multiset  obtained from   $\{b_1,\ldots, b_d\}$ by removing the elements equal to $q-1$ or $q$.
\end{lem}

\begin{proof}
The proof is the analogous to that for  Lemma  \ref{lem-4.4-1}, and so is omitted. 
\end{proof}

It should be pointed out  that the subest of $\{a_1,\ldots, a_d\}$ consisting of  $q-1$ and $q$ is not necessarily equal to the subest of $\{b_1,\ldots, b_d\}$ consisting of  $q-1$ and $q$, because of the existence of the $k$-th region (namely, the region with   topmost crossings in row $q-1$). 


We construct $(a_1',\ldots, a_d')$ by reordering  the elements of  $(a_1,\ldots, a_d)$ as follows. First, shuffle the subword including elements not equal to $q-1$ or $q$ and the subword including  $q-1$ and $q$,  such that the elements not equal to  $q-1$ or $q$ occupy the same positions as in  $(b_1,\ldots, b_d)$ (this is well defined according to Lemma \ref{huah-1}). Then, rearrange the elements not equal to $q-1$ or $q$ such that they have the same order as in  $(b_1,\ldots, b_d)$. 
For example, assuming $(b_1,\ldots, b_d)=(7,8,9,8,7,9,8,10,7,8)$, the sequence   $(a_1,\ldots, a_d)$ in Figure \ref{okp-1} will be  changed into $(a_1',\ldots, a_d')$ as illustrated  in  Figure \ref{iunf-323}.
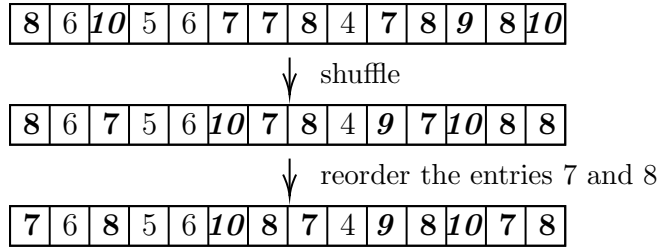
\begin{figure}[h]
    \centering

\tikzset{every picture/.style={line width=0.75pt}} 

\begin{tikzpicture}[x=0.75pt,y=0.75pt,yscale=-1,xscale=1]

\draw  [draw opacity=0] (0,0) -- (280,0) -- (280,20) -- (0,20) -- cycle ; \draw   (20,0) -- (20,20)(40,0) -- (40,20)(60,0) -- (60,20)(80,0) -- (80,20)(100,0) -- (100,20)(120,0) -- (120,20)(140,0) -- (140,20)(160,0) -- (160,20)(180,0) -- (180,20)(200,0) -- (200,20)(220,0) -- (220,20)(240,0) -- (240,20)(260,0) -- (260,20) ; \draw    ; \draw   (0,0) -- (280,0) -- (280,20) -- (0,20) -- cycle ;
\draw    (141,27) -- (141,44) ;
\draw [shift={(141,46)}, rotate = 270] [color={rgb, 255:red, 0; green, 0; blue, 0 }  ][line width=0.75]    (10.93,-3.29) .. controls (6.95,-1.4) and (3.31,-0.3) .. (0,0) .. controls (3.31,0.3) and (6.95,1.4) .. (10.93,3.29)   ;
\draw  [draw opacity=0] (0,51) -- (280,51) -- (280,71) -- (0,71) -- cycle ; \draw   (20,51) -- (20,71)(40,51) -- (40,71)(60,51) -- (60,71)(80,51) -- (80,71)(100,51) -- (100,71)(120,51) -- (120,71)(140,51) -- (140,71)(160,51) -- (160,71)(180,51) -- (180,71)(200,51) -- (200,71)(220,51) -- (220,71)(240,51) -- (240,71)(260,51) -- (260,71) ; \draw    ; \draw   (0,51) -- (280,51) -- (280,71) -- (0,71) -- cycle ;
\draw    (141,78) -- (141,95) ;
\draw [shift={(141,97)}, rotate = 270] [color={rgb, 255:red, 0; green, 0; blue, 0 }  ][line width=0.75]    (10.93,-3.29) .. controls (6.95,-1.4) and (3.31,-0.3) .. (0,0) .. controls (3.31,0.3) and (6.95,1.4) .. (10.93,3.29)   ;
\draw  [draw opacity=0] (0,102) -- (280,102) -- (280,122) -- (0,122) -- cycle ; \draw   (20,102) -- (20,122)(40,102) -- (40,122)(60,102) -- (60,122)(80,102) -- (80,122)(100,102) -- (100,122)(120,102) -- (120,122)(140,102) -- (140,122)(160,102) -- (160,122)(180,102) -- (180,122)(200,102) -- (200,122)(220,102) -- (220,122)(240,102) -- (240,122)(260,102) -- (260,122) ; \draw    ; \draw   (0,102) -- (280,102) -- (280,122) -- (0,122) -- cycle ;

\draw (5,3) node [anchor=north west][inner sep=0.75pt]   [align=left] {\textbf{8}};
\draw (25,3) node [anchor=north west][inner sep=0.75pt]   [align=left] {6};
\draw (37,3) node [anchor=north west][inner sep=0.75pt]   [align=left] {\textit{\textbf{10}}};
\draw (65,3) node [anchor=north west][inner sep=0.75pt]   [align=left] {5};
\draw (85,3) node [anchor=north west][inner sep=0.75pt]   [align=left] {6};
\draw (105,3) node [anchor=north west][inner sep=0.75pt]   [align=left] {\textbf{7}};
\draw (125,3) node [anchor=north west][inner sep=0.75pt]   [align=left] {\textbf{7}};
\draw (145,3) node [anchor=north west][inner sep=0.75pt]   [align=left] {\textbf{8}};
\draw (185,3) node [anchor=north west][inner sep=0.75pt]   [align=left] {\textbf{7}};
\draw (205,3) node [anchor=north west][inner sep=0.75pt]   [align=left] {\textbf{8}};
\draw (165,3) node [anchor=north west][inner sep=0.75pt]   [align=left] {4};
\draw (155,28) node [anchor=north west][inner sep=0.75pt]   [align=left] {\small{shuffle}};
\draw (155,79) node [anchor=north west][inner sep=0.75pt]   [align=left] {\small{reorder the entries 7 and 8}};
\draw (45,54) node [anchor=north west][inner sep=0.75pt]  [color={rgb, 255:red, 0; green, 0; blue, 0 }  ,opacity=1 ] [align=left] {\textbf{7}};
\draw (25,54) node [anchor=north west][inner sep=0.75pt]   [align=left] {6};
\draw (5,54) node [anchor=north west][inner sep=0.75pt]  [color={rgb, 255:red, 0; green, 0; blue, 0 }  ,opacity=1 ] [align=left] {\textbf{8}};
\draw (65,54) node [anchor=north west][inner sep=0.75pt]   [align=left] {5};
\draw (85,54) node [anchor=north west][inner sep=0.75pt]   [align=left] {6};
\draw (100-3,54) node [anchor=north west][inner sep=0.75pt]   [align=left] {\textbf{\textit{10}}};
\draw (125,54) node [anchor=north west][inner sep=0.75pt]   [align=left] {\textbf{7}};
\draw (145,54) node [anchor=north west][inner sep=0.75pt]   [align=left] {\textbf{8}};
\draw (182,54) node [anchor=north west][inner sep=0.75pt]   [align=left] {\textbf{\textit{9}}};
\draw (205,54) node [anchor=north west][inner sep=0.75pt]   [align=left] {\textbf{7}};
\draw (165,54) node [anchor=north west][inner sep=0.75pt]   [align=left] {4};
\draw (225-3,3) node [anchor=north west][inner sep=0.75pt]   [align=left] {\textbf{\textit{9}}};
\draw (245,3) node [anchor=north west][inner sep=0.75pt]   [align=left] {\textbf{8}};
\draw (260-3,3) node [anchor=north west][inner sep=0.75pt]   [align=left] {\textbf{\textit{10}}};
\draw (220-3,54) node [anchor=north west][inner sep=0.75pt]   [align=left] {\textbf{\textit{10}}};
\draw (265,54) node [anchor=north west][inner sep=0.75pt]   [align=left] {\textbf{8}};
\draw (45,105) node [anchor=north west][inner sep=0.75pt]  [color={rgb, 255:red, 0; green, 0; blue, 0 }  ,opacity=1 ] [align=left] {\textbf{8}};
\draw (25,105) node [anchor=north west][inner sep=0.75pt]   [align=left] {6};
\draw (5,105) node [anchor=north west][inner sep=0.75pt]  [color={rgb, 255:red, 0; green, 0; blue, 0 }  ,opacity=1 ] [align=left] {\textbf{7}};
\draw (65,105) node [anchor=north west][inner sep=0.75pt]   [align=left] {5};
\draw (85,105) node [anchor=north west][inner sep=0.75pt]   [align=left] {6};
\draw (100-3,105) node [anchor=north west][inner sep=0.75pt]   [align=left] {\textit{\textbf{10}}};
\draw (125,105) node [anchor=north west][inner sep=0.75pt]   [align=left] {\textbf{8}};
\draw (145,105) node [anchor=north west][inner sep=0.75pt]   [align=left] {\textbf{7}};
\draw (182,105) node [anchor=north west][inner sep=0.75pt]   [align=left] {\textbf{\textit{9}}};
\draw (205,105) node [anchor=north west][inner sep=0.75pt]   [align=left] {\textbf{8}};
\draw (165,105) node [anchor=north west][inner sep=0.75pt]   [align=left] {4};
\draw (265,105) node [anchor=north west][inner sep=0.75pt]   [align=left] {\textbf{8}};
\draw (217,105) node [anchor=north west][inner sep=0.75pt]   [align=left] {\textbf{\textit{10}}};
\draw (245,105) node [anchor=north west][inner sep=0.75pt]   [align=left] {\textbf{7}};
\draw (245,54) node [anchor=north west][inner sep=0.75pt]   [align=left] {\textbf{8}};

\end{tikzpicture}
\caption{An illustration from $(a_1,\ldots, a_d)$ to $(a_1',\ldots, a_d')$.}
\label{iunf-323}

\end{figure}

Define $\Phi(F)$ as the flagged filling obtained from $F^{(1)}$ by replacing the entry $a_t$  with $a_t'$ for $1\leq t\leq d$. 
Write $\Phi(F)=(F_1',\ldots, F_n')$. 
Suppose that there are $m''$  columns in the first $k-1$ regions. For $m''<j\leq n$, we clearly have $F_j'=F_j$. While, 
for $1\leq j\leq m''$, when $F_j'$ does not contain $q-1$ or $q$,  the set of entries in $F_j'$ is equal to $C_j'$. 

The construct of $\hat{\Phi}(F)$ is as follows. 
First, ignore the ``well-behaved'' columns, namely, the columns in the first $k-1$ regions of $\Phi(F)$ not containing $q-1$ or $q$. Then, merge the remaining  columns (namely, columns containing  $q-1$ or $q$)  into the  $k$-th region, in such a way that we erase all crossings above row $q-1$. Specifically, for $1\leq j\leq m''$, if $F'_j$ contains $q-1$ or $q$, then we replace $D_j$ by  $D_j\cup [q-2]=[q-2]\cup \{q\}$, and correspondingly, we replace 
$C^{(1)}_j$ and $C'_j$ respectively by $C^{(1)}_j\cup [p-2]$ and 
$C'_j\cup [p-2]$. For the portion  in Figure \ref{okp-1}, the illustration that the columns containing $q-1$ or $q$ are merged into the $k$-th region is given in  Figure \ref{pkYT-0}.
\begin{figure}[h]
    \centering

  \tikzset{every picture/.style={line width=0.6pt}} 
\begin{tikzpicture}[x=0.6pt,y=0.6pt,yscale=-1,xscale=1]

\draw  [draw opacity=0] (0,0) -- (140,0) -- (140,260) -- (0,260) -- cycle ; \draw   (20,0) -- (20,260)(40,0) -- (40,260)(60,0) -- (60,260)(80,0) -- (80,260)(100,0) -- (100,260)(120,0) -- (120,260) ; \draw   (0,20) -- (140,20)(0,40) -- (140,40)(0,60) -- (140,60)(0,80) -- (140,80)(0,100) -- (140,100)(0,120) -- (140,120)(0,140) -- (140,140)(0,160) -- (140,160)(0,180) -- (140,180)(0,200) -- (140,200)(0,220) -- (140,220)(0,240) -- (140,240) ; \draw   (0,0) -- (140,0) -- (140,260) -- (0,260) -- cycle ;
\draw    (0,159.5) -- (20,179.5) ;
\draw    (0,179.5) -- (20,159.5) ;

\draw    (0,199.5) -- (20,219.5) ;
\draw    (0,219.5) -- (20,199.5) ;

\draw    (0,219.5) -- (20,239.5) ;
\draw    (0,239.5) -- (20,219.5) ;

\draw    (0,239.5) -- (20,259.5) ;
\draw    (0,259.5) -- (20,239.5) ;

\draw    (20,159.5) -- (40,179.5) ;
\draw    (20,179.5) -- (40,159.5) ;

\draw    (40,159.5) -- (60,179.5) ;
\draw    (40,179.5) -- (60,159.5) ;

\draw    (60,159.5) -- (80,179.5) ;
\draw    (60,179.5) -- (80,159.5) ;

\draw    (20,199.5) -- (40,219.5) ;
\draw    (20,219.5) -- (40,199.5) ;

\draw    (40,199.5) -- (60,219.5) ;
\draw    (40,219.5) -- (60,199.5) ;

\draw    (80,239.5) -- (100,259.5) ;
\draw    (80,259.5) -- (100,239.5) ;

\draw    (80,159.5) -- (100,179.5) ;
\draw    (80,179.5) -- (100,159.5) ;

\draw    (100,159.5) -- (120,179.5) ;
\draw    (100,179.5) -- (120,159.5) ;

\draw    (120,159.5) -- (140,179.5) ;
\draw    (120,179.5) -- (140,159.5) ;

\draw    (20,219.5) -- (40,239.5) ;
\draw    (20,239.5) -- (40,219.5) ;

\draw    (40,219.5) -- (60,239.5) ;
\draw    (40,239.5) -- (60,219.5) ;

\draw    (60,219.5) -- (80,239.5) ;
\draw    (60,239.5) -- (80,219.5) ;

\draw    (20,239.5) -- (40,259.5) ;
\draw    (20,259.5) -- (40,239.5) ;

\draw    (40,239.5) -- (60,259.5) ;
\draw    (40,259.5) -- (60,239.5) ;

\draw    (60,239.5) -- (80,259.5) ;
\draw    (60,259.5) -- (80,239.5) ;

\draw    (100,239.5) -- (120,259.5) ;
\draw    (100,259.5) -- (120,239.5) ;

\draw    (120,239.5) -- (140,259.5) ;
\draw    (120,259.5) -- (140,239.5) ;

\draw  [line width=2.25]  (0,0) -- (60,0) -- (60,160) -- (0,160) -- cycle ;

\draw (-2+64,2.5) node [anchor=north west][inner sep=0.6pt]   [align=left] {1};
\draw (-2+84,2.5) node [anchor=north west][inner sep=0.6pt]   [align=left] {1};
\draw (-2+104,2.5) node [anchor=north west][inner sep=0.6pt]   [align=left] {1};
\draw (-2+124,2.5) node [anchor=north west][inner sep=0.6pt]   [align=left] {1};
\draw (-2+65,22.5) node [anchor=north west][inner sep=0.6pt]   [align=left] {2};
\draw (-2+85,22.5) node [anchor=north west][inner sep=0.6pt]   [align=left] {2};
\draw (-2+105,22.5) node [anchor=north west][inner sep=0.6pt]   [align=left] {2};
\draw (-2+125,22.5) node [anchor=north west][inner sep=0.6pt]   [align=left] {2};
\draw (-2+125,42.5) node [anchor=north west][inner sep=0.6pt]   [align=left] {3};
\draw (-2+125,62.5) node [anchor=north west][inner sep=0.6pt]   [align=left] {4};
\draw (-2+125,82.5) node [anchor=north west][inner sep=0.6pt]   [align=left] {5};
\draw (-2+125,102.5) node [anchor=north west][inner sep=0.6pt]   [align=left] {6};
\draw (-2+125,121.5) node [anchor=north west][inner sep=0.6pt]   [align=left] {7};
\draw (-2+125,141.5) node [anchor=north west][inner sep=0.6pt]   [align=left] {8};
\draw (-2+105,42.5) node [anchor=north west][inner sep=0.6pt]   [align=left] {3};
\draw (-2+105,62.5) node [anchor=north west][inner sep=0.6pt]   [align=left] {4};
\draw (-2+105,82.5) node [anchor=north west][inner sep=0.6pt]   [align=left] {5};
\draw (-2+105,102.5) node [anchor=north west][inner sep=0.6pt]   [align=left] {6};
\draw (-2+105,121.5) node [anchor=north west][inner sep=0.6pt]   [align=left] {7};
\draw (-2+105,141.5) node [anchor=north west][inner sep=0.6pt]   [align=left] {8};
\draw (-2+85,42.5) node [anchor=north west][inner sep=0.6pt]   [align=left] {3};
\draw (-2+85,62.5) node [anchor=north west][inner sep=0.6pt]   [align=left] {4};
\draw (-2+85,82.5) node [anchor=north west][inner sep=0.6pt]   [align=left] {5};
\draw (-2+85,102.5) node [anchor=north west][inner sep=0.6pt]   [align=left] {6};
\draw (-2+85,121.5) node [anchor=north west][inner sep=0.6pt]   [align=left] {7};
\draw (-2+85,141.5) node [anchor=north west][inner sep=0.6pt]   [align=left] {8};
\draw (-2+65,42.5) node [anchor=north west][inner sep=0.6pt]   [align=left] {3};
\draw (-2+65,62.5) node [anchor=north west][inner sep=0.6pt]   [align=left] {4};
\draw (-2+65,82.5) node [anchor=north west][inner sep=0.6pt]   [align=left] {5};
\draw (-2+65,102.5) node [anchor=north west][inner sep=0.6pt]   [align=left] {6};
\draw (-2+65,121.5) node [anchor=north west][inner sep=0.6pt]   [align=left] {7};
\draw (-2+65,141.5) node [anchor=north west][inner sep=0.6pt]   [align=left] {8};
\draw (-2+61,182.5) node [anchor=north west][inner sep=0.6pt]  [color={rgb, 255:red, 0; green, 0; blue, 0 }  ,opacity=1 ] [align=left] {10};
\draw (-2+41,182.5) node [anchor=north west][inner sep=0.6pt]  [color={rgb, 255:red, 0; green, 0; blue, 0 }  ,opacity=1 ] [align=left] {10};
\draw (-2+25,182.5) node [anchor=north west][inner sep=0.6pt]  [color={rgb, 255:red, 0; green, 0; blue, 0 }  ,opacity=1 ] [align=left] {9};
\draw (-2+81,182.5) node [anchor=north west][inner sep=0.6pt]  [color={rgb, 255:red, 0; green, 0; blue, 0 }  ,opacity=1 ] [align=left] {10};
\draw (-2+105,182.5) node [anchor=north west][inner sep=0.6pt]  [color={rgb, 255:red, 0; green, 0; blue, 0 }  ,opacity=1 ] [align=left] {9};
\draw (-2+61,202.5) node [anchor=north west][inner sep=0.6pt]  [color={rgb, 255:red, 0; green, 0; blue, 0 }  ,opacity=1 ] [align=left] {11};
\draw (-2+81,202.5) node [anchor=north west][inner sep=0.6pt]  [color={rgb, 255:red, 0; green, 0; blue, 0 }  ,opacity=1 ] [align=left] {11};
\draw (-2+101,202.5) node [anchor=north west][inner sep=0.6pt]  [color={rgb, 255:red, 0; green, 0; blue, 0 }  ,opacity=1 ] [align=left] {11};
\draw (-2+81,222.5) node [anchor=north west][inner sep=0.6pt]  [color={rgb, 255:red, 0; green, 0; blue, 0 }  ,opacity=1 ] [align=left] {12};
\draw (-2+101,222.5) node [anchor=north west][inner sep=0.6pt]  [color={rgb, 255:red, 0; green, 0; blue, 0 }  ,opacity=1 ] [align=left] {10};
\draw (-2+121,182.5) node [anchor=north west][inner sep=0.6pt]  [color={rgb, 255:red, 0; green, 0; blue, 0 }  ,opacity=1 ] [align=left] {10};
\draw (-2+125,202.5) node [anchor=north west][inner sep=0.6pt]  [color={rgb, 255:red, 0; green, 0; blue, 0 }  ,opacity=1 ] [align=left] {9};
\draw (-2+121,222.5) node [anchor=north west][inner sep=0.6pt]  [color={rgb, 255:red, 0; green, 0; blue, 0 }  ,opacity=1 ] [align=left] {12};
\draw (-2+4,2.5) node [anchor=north west][inner sep=0.6pt]   [align=left] {1};
\draw (-2+24,2.5) node [anchor=north west][inner sep=0.6pt]   [align=left] {1};
\draw (-2+44,2.5) node [anchor=north west][inner sep=0.6pt]   [align=left] {1};
\draw (-2+5,22.5) node [anchor=north west][inner sep=0.6pt]   [align=left] {2};
\draw (-2+25,22.5) node [anchor=north west][inner sep=0.6pt]   [align=left] {2};
\draw (-2+45,22.5) node [anchor=north west][inner sep=0.6pt]   [align=left] {2};
\draw (-2+45,42.5) node [anchor=north west][inner sep=0.6pt]   [align=left] {3};
\draw (-2+45,62.5) node [anchor=north west][inner sep=0.6pt]   [align=left] {4};
\draw (-2+45,82.5) node [anchor=north west][inner sep=0.6pt]   [align=left] {5};
\draw (-2+45,102.5) node [anchor=north west][inner sep=0.6pt]   [align=left] {6};
\draw (-2+45,121.5) node [anchor=north west][inner sep=0.6pt]   [align=left] {7};
\draw (-2+45,141.5) node [anchor=north west][inner sep=0.6pt]   [align=left] {8};
\draw (-2+25,42.5) node [anchor=north west][inner sep=0.6pt]   [align=left] {3};
\draw (-2+25,62.5) node [anchor=north west][inner sep=0.6pt]   [align=left] {4};
\draw (-2+25,82.5) node [anchor=north west][inner sep=0.6pt]   [align=left] {5};
\draw (-2+25,102.5) node [anchor=north west][inner sep=0.6pt]   [align=left] {6};
\draw (-2+25,121.5) node [anchor=north west][inner sep=0.6pt]   [align=left] {7};
\draw (-2+25,141.5) node [anchor=north west][inner sep=0.6pt]   [align=left] {8};
\draw (-2+5,42.5) node [anchor=north west][inner sep=0.6pt]   [align=left] {3};
\draw (-2+5,62.5) node [anchor=north west][inner sep=0.6pt]   [align=left] {4};
\draw (-2+5,82.5) node [anchor=north west][inner sep=0.6pt]   [align=left] {5};
\draw (-2+5,102.5) node [anchor=north west][inner sep=0.6pt]   [align=left] {6};
\draw (-2+5,121.5) node [anchor=north west][inner sep=0.6pt]   [align=left] {7};
\draw (-2+5,141.5) node [anchor=north west][inner sep=0.6pt]   [align=left] {8};
\draw (-2+1,182.5) node [anchor=north west][inner sep=0.6pt]  [color={rgb, 255:red, 0; green, 0; blue, 0 }  ,opacity=1 ] [align=left] {10};

\end{tikzpicture}
\caption{The merging procedure for Figure \ref{okp-1}.}
\label{pkYT-0}
\end{figure}

We remark that all the entries above row $q-1$ in the merged columns, which are framed by lines in blodface in Figure \ref{pkYT-0},  will keep unchanged in the next round of the iteration.  So the merging operation merely plays a role  that the columns in the first $k-1$ regions, containing  $q-1$ or $q$, are viewed as columns in the $k$-th region, so that our algorithms could be implemented in the next iteration.

\subsection{The first region of $D$ is of Type (R 3)}\label{subsection-4.3}

In this case, the construction of $\Phi(F)$ as well as $\hat{\Phi}(F)$  is nearly the same as that for the Type (R 2) case in Subsection \ref{sub4.2-1}. So the description  will be sketched. 
All but one of the notation (namely, the notation  $p$)  will be fully consistent with what we used in Subsection \ref{sub4.2-1}. 

Let $(D_1,\ldots, D_m)$ be the first region of $D$. 
Keep in mind that $D_j$ for $1\leq j<m$ are Type I columns, and $D_m$ is a Type III column.
For $1\leq j\leq m$, let $n_j$ be the row index such that $(n_j, j)$ is the box right above  the second crossing in column $j$.
 Note that $[n_j]\setminus C_j$ (respectively, $[n_j]\setminus C_j'$) contains a single element, or is equal to $\emptyset$ (this possibly occurs only when $j=m$), which is defined as the label $\ell_j$ (respectively, $\ell_j'$).

Perform the first algorithm as in Subsection \ref{sub4.2-1} to interchange the columns of $F$ in the first region.  
The resulting flagged filling is denoted $F^{(1)}=(F^{(1)}_1,\ldots, F^{(1)}_n)$. Suppose that $F^{(1)}$ belongs  to $\mathcal{F}_{D}(C^{(1)})$ where $C^{(1)}=(C^{(1)}_1,\ldots, C^{(1)}_n)$.  Then  
$C^{(1)}_j=C_j'$ for  $1\leq j<m$.
Moreover, one has $F^{(1)}_j=F_j$ for  $m< j\leq n$. 
We next deal with the column $F^{(1)}_m$, parallel to what we do in the second algorithm  in Subsection \ref{sub4.2-1}.

Since $D_m$ is a Type III column,  there are at least two boxes of $D_m$ lying below the second crossing. Let $p<q$ be the row indices  such that   $(p,m)$ and $(q,m)$ are the lowest two boxes of $D_m$. Unlike in Subsection \ref{sub4.2-1}, we will no longer have the relation $p=n_m$.   Suppose that there are   $k$ regions whose first crossings lie in or above row $q-1$. 

For $1\leq r\leq k$, assume that the first crossing in the $r$-th region  is in row $i_r$.
We consider two cases.

Case 1. $i_k<q-1$. The left picture in Figure \ref{PTED-1221} is an instance of this case.  Define  $\Phi(F)$  and $\hat{\Phi}(F)$ by applying the same procedure as in Case 1 of Subsection \ref{sub4.2-1}.
\begin{figure}[h]
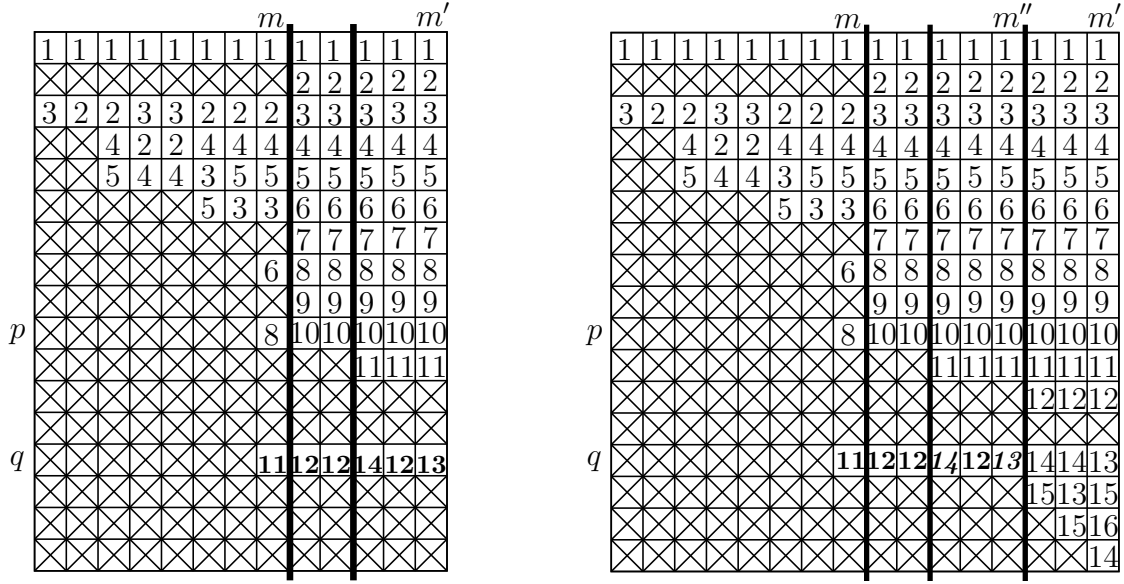

\subfigure{
\begin{minipage}[t]{7cm}
 
    \centering
    \tikzset{every picture/.style={line width=0.6pt}} 


\end{minipage}

}
\caption{Illustrations for Type (R 3) case with $i_k<q-1$ or $i_k=q-1$.}\label{PTED-1221}
\end{figure}

Case 2. $i_k=q-1$. This case is demonstrated in the right picture in  Figure \ref{PTED-1221}.
Define  $\Phi(F)$  and $\hat{\Phi}(F)$ by applying  the same procedure as in Case 2 of Subsection \ref{sub4.2-1}.

\subsection{Proof of Theorem \ref{main-bijection}}

Starting with $F\in \mathcal{F}_D(C)$, we iterate the operation $\Phi$ and  eventually arrive at a flagged filling, denoted $\Omega(F)$, in  $\mathcal{F}_D(C')$. We explain  that $\Omega$ is  a bijection from $\mathcal{F}_D(C)$ to $\mathcal{F}_D(C')$
that preserves both the sign and weight.
To check that $\Omega$ is a bijection, the key is to make clear  which and how   columns are swapped in each step.
By our construction, this only depends on the diagrams $D, C$ and $ C'$, independent of the flagged filling $F$. Moreover, it is easy to see that  the operation in each step in Subsections \ref{subset4.1}, \ref{sub4.2-1}, \ref{subsection-4.3} may be reversed. 

We next check that $\Omega$ preserves the sign and weight of $F$. It suffices to verify that $F$ and $\Phi(F)$ have the same sign and weight. Since the entries (if moved) are slid  in the same row, we have $y^{F}=y^{\Phi(F)}$. 
To see that $F$ and $\Phi(F)$ have the same sign, recall that there are two kinds of operations in our construction.
\begin{itemize}
\item[(1)] Two columns $F_{j_1}$ and $F_{j_2}$ ($j_1<j_2$) of $F$ are exchanged. Suppose that $F_{j_1}$ has column  reading word  $u=a_1 a_2\cdots a_s$, and $F_{j_1}$ has column  reading word $v=b_1 b_2\cdots b_t$. Note that $s\leq t$.
After column-exchanging, the reading words in columns $j_1$ and $j_2$ become $u'=b_1 b_2\cdots b_s$ and $v'=a_1 a_2\cdots a_s \,b_{s+1}\cdots b_t$, respectively. Notice that any element in $\{b_{s+1}, \ldots, b_t\}$ is bigger than any element in $\{a_1,\ldots, a_s\}$ or   $\{b_1,\ldots, b_s\}$.
This implies that $\mathrm{inv}(u)+\mathrm{inv}(v)=\mathrm{inv}(u')+\mathrm{inv}(v')$.

\item[(2)] Some entries in row $q$ are reordered. In this case, the inversion  number of each column reading word is unchanged since the moved entries in row $q$ are bigger than any entry above row $q$ in the corresponding columns.
\end{itemize}

By the  above analysis, we obtain that $\Omega$ is a sign- and weight-preserving bijection. This allows us to conclude  the proof of \ref{main-bijection}.

\bigbreak

\footnotesize{

\textsc{(Peter L. Guo) Center for Combinatorics, Nankai University, LPMC, Tianjin 300071, P.R. China}

{\it
Email address: \tt lguo@nankai.edu.cn}

\medbreak

\textsc{(Zhuowei Lin) Center for Combinatorics, Nankai University, LPMC, Tianjin 300071, P.R. China}

{\it
Email address: \tt zwlin0825@163.com}

\medbreak

\textsc{(Simon C.Y. Peng) Center for Applied Mathematics, Tianjin University, Tianjin 300072, P.R. China}

{\it
Email address: \tt pcy@tju.edu.cn}

}
\end{document}